\documentclass[reqno,a4paper]{amsart}

\usepackage[round]{natbib}
\usepackage[english]{babel}
\usepackage{dsfont}
\usepackage{amssymb}
\usepackage{amsmath}
\usepackage{amsfonts}
\usepackage{slashbox}
\usepackage{tabularx}
\usepackage{graphicx}
\usepackage{url}

\usepackage{color}

\newtheorem{assumption}{Assumption}[section]

\newtheorem{thm}{Theorem}[section]
\newtheorem{lem}[thm]{Lemma}
\newtheorem{remark}{Remark}[section]
\newtheorem{proposition}{Proposition}[section]

\newcommand{\abs}[1]{\lv #1\rv}

\newcommand{\di}{\,\mrm{d}}

\newcommand{\ds}{\,{\mrm{d}}s}
\newcommand{\dt}{\,{\mrm{d}}t}

\newcommand{\de}{\delta}
\newcommand{\dn}{\delta_n}
\newcommand{\dsp}{\displaystyle}

\newcommand{\eqv}{\Longleftrightarrow}
\newcommand{\esp}{\mathds{E}}

\newcommand{\gm}{\gamma}

\newcommand{\indi}{\mathds{1}}
\newcommand{\p}{\mathds{P}}
\newcommand{\q}{\mathds{Q}}

\newcommand{\lv}{\left\vert}
\newcommand{\mrm}{\mathrm}
\newcommand{\n}{{\mathds{N}}}
\newcommand{\norm}[1]{\left\Vert#1\right\Vert}

\newcommand{\reel}{{\mathds{R}}}

\newcommand{\rv}{\right\vert}

\newcommand{\tht}{\theta}
\newcommand{\tv}{\xrightarrow}
\newcommand{\var}{{\mrm{Var\,}}}
\newcommand{\ve}{\varepsilon}

\newcommand{\vp}{\varphi}

\newcommand{\z}{{\mathds{Z}}}

\newcommand{\cal}{\mathcal}

\begin{document}

\title{Bayesian prediction for stochastic processes. Theory and applications}

\author[D. Blanke]{Delphine Blanke}

\address{Avignon University,  Laboratoire de Math\'{e}matiques d'Avignon,  33 rue Louis Pasteur, 84000 Avignon, France} 
\email{delphine.blanke@univ-avignon.fr}

\author[D. Bosq]{Denis Bosq}

\address{Sorbonne Universit\'{e}s, UPMC Univ Paris 06, EA 3124, Laboratoire de Statistique Th\'{e}orique et Appliqu\'{e}e, F-75005, 4 place Jussieu, Paris, France} \email{denis.bosq@upmc.fr}

\begin{abstract} In this paper, we adopt a Bayesian point of view for predicting real
continuous-time processes. We give two equivalent definitions of a Bayesian predictor and study some
properties: admissibility, prediction sufficiency, non-unbiasedness, comparison with
efficient predictors.  Prediction of Poisson process and prediction of Ornstein-Uhlenbeck process
in the continuous and sampled situations are considered. Various simulations illustrate comparison with non-Bayesian predictors.
\end{abstract}

\subjclass[2000]{Primary 62M20, 62F15}
\keywords{Bayesian prediction, MAP, Comparing predictors, Poisson process, Ornstein-Uhlenbeck process}

\maketitle

\hyphenation{si-tuation}

\section{Introduction}\label{s1}

A lot of papers are devoted to Bayesian estimation for stochastic processes (see for example \citet{Ku04} for the asymptotic point of view) while Bayesian prediction does not appear very much in statistical literature.  Some authors have studied the case of linear processes  \citep[see][]{Di90,SM00,SM03} but continuous time is not often considered. However, this topic is important, in particular if the number of data is small. In this paper, we study some properties of Bayesian predictors and give examples of applications to prediction of continuous-time processes.  Note that we don't consider prediction for the linear model, a somewhat different topic which has been extensively studied in literature. In fact, our main goal is to compare efficiency of Bayesian predictors with non-Bayesian ones, especially if we have few data at our disposal. Various simulations illustrate the obtained results.

Section~\ref{s2} presents the general prediction model~;~in this context estimation appears as a special case of prediction. The main point of the theory is the fact that, given the data $X$, a statistical predictor of $Y$ is an approximation of the conditional expectation $\esp_{\tht}(Y|X)$, where $\tht$ is the unknown parameter. Section~\ref{s3} deals with Bayesian prediction:~we give two equivalent definitions of a Bayesian predictor linked with the equivalence of predicting $Y$ and $\esp_{\tht}(Y|X)$. However, in some situations, it is difficult to get an explicit form of the Bayesian predictor, thus it is more convenient to substitute the conditional expectation with the conditional mode. An alternative method consists in computing the Bayesian estimator or the maximum a posteriori (MAP)  and to plug it in $\esp_{\tht}(Y|X)$. We recall some properties of the MAP and underscore its link with the maximum likelihood estimator.

In Section~\ref{s4}, we study some properties of Bayesian predictors: admissibility, connection with sufficiency and unbiasedness, case where the conditional expectation admits a special form. Section~\ref{s5} considers the simple case of Poisson process prediction. We compare the unbiased efficient predictor with the Bayesian and the MAP ones. Concerning diffusion processes, note that \citet{TV05} obtain fine results for Bayesian prediction but without comparison with classical predictors.

For the Ornstein-Uhlenbeck process, we deal with prediction in Section~\ref{s6} for the centered and non-centered case and with various priors, while Section~\ref{s7} is devoted to the sampled case. Some asymptotic results are given along the paper, but, since the non asymptotic case is the most important in the Bayesian perspective, theoretical and numerical comparisons focus on this point.

\section{The prediction model}\label{s2}

In the non Bayesian context, let $(X,Y)$ be a random vector, defined on some Probability space and with values in a measurable space $(F \times G, {\cal F} \otimes {\cal G})$. In the following, $F$ and $G$ will be $\n$ or $\reel^k$, $k\ge 1$. $(X,Y)$ has distribution $(\p_{\tht}, \tht \in \Theta)$ where $\theta$ is the unknown parameter and $\Theta$ is an open set in $\reel$. We suppose that $\p_{\tht}$ has a density $f(x,y,\tht)$ with respect to a $\sigma$-finite measure $\lambda \otimes \mu$.

One observes $X$ and wants to predict $Y$. Actually, it is possible to consider the more general problem `\emph{predict $Z=\ell(X,Y,\tht)$ given $X$}'  \citep[cf.][]{Ya92}.

In this paper, we suppose that $Z$ is real valued and denote $p(X)$ (or $q(X)$) a statistical predictor. Then, if $Z$, $p(X)$ and $q(X)$ are square integrable,  a classical preference relation is
$$p(X)\prec q(X) \;\;(Z) \Longleftrightarrow \esp_{\tht}(p(X)-Z)^2 \le \esp_{\tht}(q(X)-Z)^2, \; \tht\in\Theta$$
where `(Z)' means `for predicting $Z$' and $\esp_{\tht}$ is the expectation taken with respect to the distribution $\p_{\tht}$.

Now, let $\esp_{\tht} (Z | X)$ be the conditional expectation of $Z$ given $X$ associated with the distribution $\p_{\tht}$. The next lemma is simple but important.
\begin{lem}\label{l1} We have $p(X) \prec q(X) \;\; (Z) \Longleftrightarrow p(X) \prec q(X) \;\;\big(\esp_{\tht}(Z|X)\big)$.
\end{lem}

\begin{proof} The result directly follows from the Pythagoras theorem, since:
$$
\esp_{\tht} ( p(X) - Z)^2 = \esp_{\tht} \big( p(X) - \esp_{\tht}(Z|X)\big)^2 +  \esp_{\tht} (\esp_{\tht}(Z|X)-Z)^2
$$
and $$\esp_{\tht} ( q(X) - Z)^2 = \esp_{\tht} \big( q(X) - \esp_{\tht}(Z|X)\big)^2 +  \esp_{\tht} (\esp_{\tht}(Z|X)-Z)^2. $$ \end{proof}

\indent This lemma shows that predicting $Z$ or predicting $\esp_{\theta}(Z| X)$ is the same problem.

Note that prediction theory has some similarity but also some difference with estimation theory. In the sequel, we will only recall some necessary definitions and results. We refer to \citet{BB07}, chapters 1 and 2, for a more complete exposition.

\section{Bayesian prediction}\label{s3}
\subsection{The Bayesian predictor}
In the Bayesian framework, we suppose that $\mathbb{T}$ is a random variable with prior distribution $\tau$ over $\Theta$, and admitting a density $\vp(\tht)$ with respect to a $\sigma$-finite measure $\nu$ \citep[cf.][]{LC98}.

Thus, we may consider the scheme
$$(\Omega, {\cal A}, \p) \tv[]{(X,Y, \mathbb{T})} (F \times G \times \Theta, {\cal F} \otimes {\cal G} \otimes {\cal B}_{\theta}),$$
where $(\Omega, {\cal A}, \p)$ is a probability space, $(X,Y, \mathbb{T})$  a random vector and ${\cal B}_{\tht}$ the $\sigma$-algebra of Borel sets over $\Theta$.  Now, we denote $\mathds{Q}$ the distribution of $(X,Y,\mathbb{T})$ and we  consider the following regularity assumption:
\begin{assumption} \label{A1}\hfill\\
$\mathds{Q}$ admits a strictly positive density $f(x,y,\theta)\vp(\theta)$  over $F\times G\times \Theta$,  with respect to the $\sigma$-finite measure $\lambda \otimes \mu \otimes \nu$. In addition, $f$ and $\vp$ are supposed to be continuous with respect to $\tht$ on $\Theta$.
\end{assumption}
Note that in practice, $\lambda$, $\mu$ and $\nu$ can be the Lebesgue measure or the counting measure. Also, remark that similar results can be derived under a more general version of Assumption~\ref{A1}, namely the existence of a common version $m(X,\tht)$ of $\esp_{\tht}(Y|X)$ for all $\tht\in \Theta$  \citep[see][]{BB12}.

\medskip

Now, the Bayesian risk  for prediction is
$$r(p(X),Y) := \esp(p(X) - Y)^2= \int_{\Theta} \esp_{\tht} (p(X) - Y)^2 \vp(\theta) \mathrm{d}\nu(\theta),$$
where $\esp_{\tht}$ is expectation taken with respect to $\p_{\tht}$ and $\esp$ is expectation taken with respect to $\mathds{Q}$.

It follows that the Bayesian predictor is
\begin{equation}\label{e31}
p_0(X) =\arg\!\min_p r(p(X),Y) = \esp(Y|X).
\end{equation}
More precisely, we choose $p(X)$ under the form
$$p_0(X) = \int_G y f(y|X) \mathrm{d}\mu(y)$$
where
$$f(y|X) = \frac{\int_{\Theta} f(X,y,\theta) \vp(\theta) \mathrm{d}\nu(\tht)}{\int_{G \times \Theta} f(X,y,\tht) \vp(\tht) \mathrm{d}\mu(y) \mathrm{d}\nu(\tht)}$$
which ensures existence and uniqueness of $p_0(X)$ under Assumption~\ref{A1}. In the following, we set
$$m(X,\tht) = \esp_{\tht}(Y| X)= \int_G y f_{\theta}(y|x) \mathrm{d} \mu(y), \tht\in\Theta$$
where
$$f_{\tht}(y|x) = \frac{f(x,y,\tht)}{\int_F f(x,y,\tht) \mathrm{d} \lambda(x)}.$$
\begin{remark} \label{r31} If Assumption~\ref{A1} holds, the relation $\esp(Y| X) = \esp \big( \esp(Y| X,\mathbb{T}) \big| X\big)$ gives   the following alternative form of $p_0$:  $$p_0(X)= \esp \big( m(X,\mathbb{T})|X\big)$$ where $m(X,\mathbb{T}) = \esp \big( Y| X,\mathbb{T} \big)$.
\end{remark}

\subsection{The MAP predictor}

An alternative method of Bayesian prediction is based on the conditional mode: one may compute the mode of the distribution of $Y$, given $X$, with respect to $\mathds{Q}$. If a strictly positive density does exist, the distribution of $(X,Y)$ has marginal density $f(x,y) = \int_{\Theta} f(x,y,\tht) \vp(\tht) \di\tht$
and, in fact, it suffices to compute $\dsp \arg\!\max_y f(x,y)$ ($x$ fixed).

A related method consists in determining the mode of $\mathbb{T}$ given $X$ and to plug it in the conditional expectation $\esp_{\tht}(Y|X)$. This mode (also called \emph{maximum a posteriori}, MAP) has the expression
$$
\widetilde{\tht}(x)= \arg\!\max_{\tht} \frac{\ell(x,\tht)\vp(\tht)}{\int_{\Theta} \ell(x,\tht)\vp(\tht) \di\tau(\tht) } = \arg\!\max_{\tht} \ell(x,\tht)\vp(\tht)
$$
where $\ell(x,\tht) = \int f(x,y,\tht) \di\mu(y)$, hence the MAP predictor
$$\widetilde{p}(X) = \esp_{\tht} (Y|X)\, \Big\vert_{\tht=\widetilde{\tht}(X)} = m(X,\widetilde{\tht})$$
under  Assumption~\ref{A1}. It is noteworthy that, if $\Theta=\reel$ and one chooses the improper prior $1\cdot \lambda$, where $\lambda$ is Lebesgue measure, the obtained estimator is the maximum likelihood (MLE). Note also that, if $\ell(x,\tht)\vp(\tht)$ is symmetric with respect to $\widetilde{\theta}(X)$, the MAP and the Bayes estimator of $\tht$ coincide. Finally, it is clear that, under classical regularity conditions, the MAP and the MLE have the same asymptotic behaviour as well almost surely as in distribution. Now, the MAP has some drawbacks: it is often difficult to compute and uniqueness is not guaranteed. We will use the MAP in Sections \ref{s5} to \ref{s7}.

\section{Properties of Bayesian predictors}\label{s4}
We give below some useful properties of Bayesian predictors. Here, we suppose that $p_0(X)$ does exist and is defined by relation \eqref{e31}.
\subsection{Admissibility}
A Bayesian predictor is said to be unique if it differs, for any other Bayesian predictor, only on a set ${\cal N}$ with $\p_{\tht}({\cal N}) = 0$ for all $\theta\in\Theta$ \citep[see][p. 323]{LC98}. Then, we have
\begin{proposition} \label{p1}
A Bayesian predictor is admissible as soon as it is unique.
\end{proposition}
\begin{proof}
 If $p_0(X)$ is not admissible, there exists a predictor $p(X)$ such that
$$\esp_{\tht}(p(X)-Y)^2 \le\esp_{\tht}(p_0(X) - Y)^2, \; \tht\in \Theta.$$
Integrating with respect to $\tau$ entails $r(p(X),Y)\le r(p_0(X),Y)$,  but, since $p_0(X)$ is Bayesian, it follows that
$r(p(X),Y) = r(p_0(X),Z)$ and uniqueness of $p_0(X)$ gives
$p(X)=p_0(X) \;\; (\p_{\tht}$  a.s. for all  $\tht$).  \end{proof}
\subsection{$Y$-Sufficiency}
A statistic $S=S(X)$ is said to be $Y$-sufficient (or sufficient for predicting $Y$) if
\begin{itemize}
\item[(a)] $S$ is sufficient in the statistical model associated with $X$: there exists a version of the conditional distribution of $X$ given $S$, say $\mathds{Q}^S$, that does not depend on $\tht$.
\item[(b)] $X$ and  $Y$ are conditionally independent given $S$.
\end{itemize}
Note that this does not imply that $\esp_{\tht}\big(Y|S(X)\big)$ is constant with respect to $\tht$ since the sufficient statistic is in the submodel generated by  $X$ (see example of the Poisson process in Section \ref{s5}). If $S$ is $Y$-sufficient, it is then possible to derive a Rao-Blackwell theorem as well as a factorization theorem  \citep[cf.][]{BB07}. Now, we have
\begin{lem}\label{l3}
If $S$ is $Y$-sufficient and Assumption~\ref{A1} holds, then
\begin{equation}\label{e41}
\esp_{\tht}(Y|X) = \esp_{\tht} \big( Y|S(X)\big), \;\; \tht\in\Theta.
\end{equation}
\end{lem}
\begin{proof} We have $$\esp_{\tht}\big(Y|S(X)\big)= \esp_{\tht}\Big(\esp_{\tht}(Y|X) \big|S(X)\Big),$$ and applying (b) to $Y$ and $\esp_{\tht}(Y|X)$ we obtain
$$
\esp_{\tht}\big(Y \cdot \esp_{\tht}(Y|X)\big|S(X)\big) = \esp_{\tht}\big(Y|S(X)\big) \cdot \esp_{\tht}\Big( \esp_{\tht}(Y|X)\big| S(X)\big)= \Big( \esp_{\tht}\big(Y|S(X)\big)\Big)^2.
$$
Taking expectation and noting that
$\esp_{\tht} \big(Y\cdot \esp_{\tht}(Y|X)\big) =\esp_{\tht}\Big( \big( \esp_{\tht}(Y|X)\big)^2\Big)$,
entails
$$\esp_{\tht}\Big( \big( \esp_{\tht}(Y|X)\big)^2\Big)= \esp_{\tht}\Big( \big( \esp_{\tht}(Y|S(X))\big)^2\Big),$$
that is $\norm{\esp_{\tht}(Y|X)}^2_{L^2(\p_{\tht})} = \norm{\esp_{\tht}\big(Y|S(X)\big)}^2_{L^2(\p_{\tht})}$. This implies relation \eqref{e41} since $\esp_{\tht}\big(Y|S(X)\big)$ is the projection of $\esp_{\tht}(Y|X)$ on $L^2_{S(X)}$.\end{proof}

Note that, if $(X,Y)$ has a strictly positive density of the form $f(x,y,\tht) = L(S(x),y,\tht)$, one obtains \eqref{e41} by a direct computation. Concerning the Bayesian predictor, we have
\begin{proposition} \label{p2}
If $p_0$ is unique and $S$ is $Y$-sufficient, then $$p_0(X) =\esp\big( Y|S(X) \big).$$
\end{proposition}
\begin{proof}  Since $p_0(X) = \esp(Y|X)$, the Rao-Blackwell theorem for prediction \citep[cf.][p. 15]{BB07} entails
$p_1(X) := \esp^{S(X)}\big( \esp(Y |X)\big) \prec p_0(X)$ where $\esp^{S(X)}$ is conditional expectation with respect to $\mathds{Q}^S$ in (a). Now, from Proposition \ref{p1}, $p_0$ is admissible, thus
$$p_0(X) =p_1(X)= \esp^{S(X)}\big( \esp(Y |X)\big) = \esp\big( Y | S(X)\big).$$  \end{proof}
\subsection{Decomposition of the conditional expectation}
We now consider the special case where the conditional expectation admits the following decomposition:
\begin{equation}\label{e42}
\esp_{\tht} (Y|X) = A(X) + B(\tht) C(X) + D(\tht), \;\;\tht\in\Theta
\end{equation}
where $A$, $B\otimes C$, $D\in L^2(F\times \Theta, {\cal F}\otimes{\cal T},\mathds{Q}_{(X,\mathbb{T})})$, $\mathds{Q}_{(X,\mathbb{T})}$ being the distribution of $(X,\mathbb{T})$. Then, the Bayesian predictor has also a special form:
\begin{proposition} \label{p3}
Suppose that Assumption~\ref{A1} is fulfilled. If $\,\esp_{\tht}(Y|X)$ satisfies \eqref{e42}, the associated Bayesian predictor is given by
\begin{equation}\label{e43}
p_0(X) = A(X) + \esp(B(\mathbb{T})|X)\cdot C(X) +\esp(D(\mathbb{T})|X).
\end{equation}
In particular, if $X$ and $Y$ are independent and $D(\tht) = \esp_{\tht}(Y)$, the predictor reduces to the estimator
$p_0(X) = \esp(D(\mathbb{T})|X)$.

\end{proposition}
\begin{proof} Relation \eqref{e42} entails
$m(X,\mathbb{T}) = A(X) + B(\mathbb{T})\cdot C(X) + D(\mathbb{T}),$
and Remark~\ref{r31} gives $p_0(X) = \esp(m(X,\mathbb{T})|X)$ hence \eqref{e43} from the properties of conditional expectation. The last assertion is a special case of \eqref{e43}.  \end{proof}
\subsection{Unbiasedness}
A predictor $p(X)$ of $Y$ is said to be unbiased if $\esp_{\tht}\, p(X) =\esp_{\tht}(Y)$, $\tht\in\Theta$. A Bayesian estimator is, in general, not unbiased, in fact we have the following:
\begin{lem}[Blackwell-Girschick]\label{l4} Let $\widehat{\vp}(X)$ be an unbiased Bayesian estimator of $\vp(\tht)$, then
$$\esp\big( \widehat{\vp}(X) -\vp(\mathbb{T})\big)^2= 0
$$
where $\esp$ denotes here expectation taken from $\q_{(X,\mathbb{T})}$.
\end{lem}
\begin{proof} See \citet[][p. 234]{LC98}.   \end{proof}
The situation is more intricate concerning a Bayesian predictor. Note first that, if
\begin{equation}\label{e44}
\esp_{\tht} (p_0(X)) = \esp_{\tht}(Y),\;\;\tht \in \Theta
\end{equation}
then, $p_0(X)$ is an unbiased estimator of $\esp_{\tht}(Y)$ but it is not necessarily a Bayesian estimator of $\esp_{\tht}(Y)$. Recall that the Bayesian interpretation of  \eqref{e44} is: $$\esp(p_0(X)|\mathbb{T}=\tht) = \esp(Y|\mathbb{T}=\tht), \;\;\tht\in\Theta.$$

Now, we have the following result:
\begin{proposition}\label{p4}
If the Bayesian risk satisfies
\begin{equation}\label{e45}
\esp\big( p_0(X) - m(X,\mathbb{T})\big)^2 = 0
\end{equation}
then $p_0(X)$ is unbiased for predicting $Y$. Conversely under Assumption~\ref{A1}, if
\begin{equation}\label{e46}
m(X,\tht)= A(X) + D(\tht),\;\tht\in\Theta
\end{equation}
and if $p_0(X)$ is unbiased, then \eqref{e45} holds.
\end{proposition}
\begin{proof} Relation \eqref{e45} implies $\dsp p_0(X) = m(X,\mathbb{T})$, $\q_{(X,\mathbb{T})}$ a.s.,
that is
$$\esp(Y|X)= \esp(Y|X,\mathbb{T})\;\;\; \q_{(X,\mathbb{T})} \;\text{ a.s.}.$$
Conditioning with respect to $\mathbb{T}$ gives $\esp(p_0(X)|\mathbb{T}) = \esp(Y|\mathbb{T})$ which means that $p_0(X)$ is unbiased. Conversely \eqref{e46} and \eqref{e43} in Proposition~\ref{p3}   imply
$$p_0(X) = A(X) + \esp(D(\mathbb{T})|X).$$
Now, since $p_0(X)$ is unbiased, we have
$$\esp(Y|\mathbb{T}) =  \esp(A(X) | \mathbb{T}) + \esp( \esp(D(\mathbb{T})/X) | \mathbb{T}) = \esp(m(X,\mathbb{T})|\mathbb{T})$$
where the last equality follows from $\esp(Y|X,\mathbb{T})=m(X,\mathbb{T})$ and a conditioning on $\mathbb{T}$. But by \eqref{e46},
$$\esp(m(X,\mathbb{T}) | \mathbb{T}) = \esp(A(X)| \mathbb{T}) + \esp(D(\mathbb{T}) |\mathbb{T}).$$
By identification, it means that the Bayesian estimator of $D(\mathbb{T})$ is also unbiased. Then Lemma~\ref{l4} gives
$$\esp\big(p_0(X) - m(X,\mathbb{T})\big)^2 = \esp\big( \esp(D(\mathbb{T}) | X) - D(\mathbb{T})\big)^2 =0. \;\;\;  $$
\end{proof}
\indent In the more general case where $\esp_{\tht}(Y|X)$ has the form \eqref{e42} with non-null $B(\tht)C(X)$, it is possible to find an unbiased Bayesian predictor with a non-vanishing Bayesian risk \citep[cf.][]{Bo12}.

Now for some $\tht_0\in\Theta$, let us define a `Bayesian type' predictor  by
\begin{equation}\label{e47}
p_0(X)= \alpha \,p(X) + (1-\alpha) m(X,\tht_0),\;\;\;(0<\alpha<1),
\end{equation}
where $p(X)$ is an unbiased predictor of $Y$. For these specific predictors, our previous result may be extended as follows.
\begin{proposition} \label{p45}
Suppose that Assumption~\ref{A1} holds and consider a predictor $p_0(X)$ of the form \eqref{e47}.  Then, if $p_0(X)$ is unbiased, it follows that
\begin{equation}\label{e48}
\esp_{\tht}\big(m(X,\tht)\big) = \esp_{\tht}\big(m(X,\tht_0)\big), \;\;\;\tht\in\Theta,
\end{equation}
if, in addition, there exists a $Y$-sufficient complete statistic then $
m(X,\tht)=m(X,\tht_0)$ for all $\tht\in\Theta$ and the problem of prediction is degenerated.
\end{proposition}
\begin{proof} If $p_0(X)$ is an unbiased predictor of $Y$, one has
$$\esp_{\tht}\big(p_0(X)\big) =\esp_{\tht} \big(m(X,\tht)\big), \;\tht\in\Theta,$$ and taking expectation in \eqref{e47} yields
$$ \esp_{\tht} \big(m(X,\tht)\big)=\alpha \, \esp_{\tht}\big(p(X)\big) + (1-\alpha) \esp_{\tht} \big(m(X,\tht_0)\big)
$$
hence, since $p(X)$ is unbiased, \eqref{e48} follows. Now, if $S(X)$ is a $Y$-sufficient statistic, Lemma~\ref{l3} entails $m(X,\tht) = \esp_{\tht}\big(Y|S(X)\big)$, thus, \eqref{e48} implies
$$\esp_{\tht}\Big( \esp_{\tht} \big(Y|S(X)\big) - \esp_{\tht_0} \big(Y|S(X)\big)\Big) = 0, \;\;\tht\in\Theta,$$
and, since $S(X)$ is complete, one obtains the last result.    \end{proof}
\subsection{Comparing predictors}
The following elementary lemma allows to compare Bayesian predictors with the classical unbiased predictor. We will use it in the next sections.
\begin{lem}\label{l4b} Suppose that
$$m(X,\tht) = A(X) + d\cdot \tht \;\;\;(d\not=0)$$ and let $p(X)$ be an unbiased predictor of $Y$ taking the form
$$p(X) = A(X) + d \cdot \overline{\tht}(X).$$  For some $\tht_0\in \Theta$, consider the `Bayesian type' predictor
$$p_0(X) = \alpha\, p(X)+(1-\alpha)m(X,\tht_0)$$
where $\alpha\in]0,1[$. Then
\begin{equation}\label{e49}
p_0 \prec p \Longleftrightarrow \abs{\tht-\tht_0} \le \Big(\frac{1+\alpha}{1-\alpha}\Big)^{\frac{1}{2}}\cdot \Big(\esp_{\tht}\big(\overline{\tht}(X) -\tht\big)^2\Big)^{\frac{1}{2}}.
\end{equation}
\end{lem}
\begin{proof} We have
$$p_0(X) -m(X,\tht)=\alpha \big( p(X) -m(X,\tht)\big)+ (1-\alpha) \big( m(X,\tht_0) -m(X,\tht)\big)$$
then, since $p$ is unbiased,
$$\esp_{\tht}\big(p_0(X) -m(X,\tht)\big)^2=\alpha^2 \esp_{\tht}\big(p(X) -m(X,\tht)\big)^2 +(1-\alpha)^2 d^2 (\tht_0 - \tht)^2$$
thus
$$p_0 \prec p \Longleftrightarrow d^2(1-\alpha)^2(\tht-\tht_0)^2+\alpha^2 \esp_{\tht} \big(p(X) - m(X,\tht)\big)^2 \le \esp_{\tht}\big(p(X) -m(X,\tht)\big)^2$$
and \eqref{e49} follows. \end{proof}

\begin{remark} If $X=X_{(n)} = (X_1,\dotsc,X_n)$ and $\esp_{\tht}\big(\overline{\tht}(X) -\tht\big)^2 = \frac{v^2}{n}$ then the condition becomes
$$\abs{\tht -\tht_0} \le  \Big( \frac{1+\alpha}{1-\alpha} \Big)^{\frac{1}{2}} \cdot \Big(\frac{v^2}{n}\Big)^{\frac{1}{2}}.$$  If one may find $\alpha=\alpha_n$ such that $$\inf\limits_{n\ge 1} \Big(\frac{1+\alpha_n}{1-\alpha_n}\Big)^{\frac{1}{2}}\cdot\Big(\frac{v^2}{n}\Big)^{\frac{1}{2}} \ge b > 0,$$
it follows that $\abs{\tht -\tht_0} \le b$ implies $p_0(X_{(n)}) \prec p(X_{(n)})$ for all $n \ge 1$. Moreover, the choice $A(X) \equiv 0$ in Lemma~\ref{l4b} provides an alternative formulation for comparing Bayesian estimators of $\tht$  versus non Bayesian ones. \end{remark}

\section{Application to Poisson process} \label{s5}

\subsection{The Bayesian predictor}

Let $(N_t,\,t\ge 0)$ be an homogeneous Poisson process with intensity $\tht>0$, $X=(N_t,\; 0\le t\le S)$ is observed and one wants to predict $Y=N_{S+h}$ $(h>0)$, $(S>0)$. This a classical scheme but of interest, since in this case, there exists an unbiased efficient predictor   \citep[see][]{BB07}. Since Lemma~\ref{l1} shows that it is equivalent to predict $m(X,\tht) = \tht h + N_S$, one obtains the unbiased efficient predictor $p(N_S) =  \frac{S+h}{S} N_S=: N_S + \tht_S h$ (with $\theta_S = \frac{N_S}{S}$).

Concerning the Bayesian predictor, a classical prior is $\tau = \Gamma(a,b)$ with density $$ \frac{b^{a}}{\Gamma(a)} \tht^{a-1} \exp(-b\tht)\indi_{]0,+\infty[}(\tht),\;\;\;(a>0,\, b>0).$$ First, since $N_S$ is $N_{S+h}$-sufficient, Lemma~\ref{l3} entails
$$\esp_{\tht}(N_{S+h} | N_t, \, 0\le t\le S) = \esp_{\tht} (N_{S+h} | N_S)$$
and Proposition~\ref{p2} gives $ p_0(N_t, \, 0\le t \le S) = \esp(N_{S+h}|N_S).$
The same property holds for the Bayes estimator given  by
$$\widehat{\tht}_S = \esp(\mathbb{T} | N_S) = \frac{a+N_S}{b+S},$$  and, from Proposition~\ref{p3}, the Bayesian predictor is $$\widehat{p}_0(N_S) = \frac{a+N_S}{b+S} \cdot h + N_S.$$

To compare $\widehat{p}_0$  with $p$, note that $\widehat{\tht}_S= \frac{S}{b+S}\cdot \tht_S + \big( 1 - \frac{S}{b+S}\big) \cdot \frac{a}{b}.$
We deduce that $$\widehat{p}_0(N_S) = \alpha_S \,p(N_S) + (1-\alpha_S) (N_S + \tht_0 h)$$ with  $\alpha_S = \frac{S}{b+S}$ and $\tht_0=\frac{a}{b}$. Since $\esp_{\tht}\big( \tht_S - \tht\big)^2 = \frac{\tht}{S}$, a straightforward consequence of Lemma~\ref{l4b} is
\begin{equation}\label{e51} \widehat{p}_0\prec p  \eqv (\tht - \tht_0)^2 \le \big(\frac{1}{S} + \frac{2}{b}\big)\tht.
\end{equation}
Solving \eqref{e51} in $\tht$, we get that $ p_0\prec p $ iff
$$\tht \in \Big]\tht_0+\frac{1}{2S} + \frac{1}{b}-\sqrt{\Delta},\tht_0+\frac{1}{2S} + \frac{1}{b}+\sqrt{\Delta}\Big[$$ with $\Delta = \big(\tht_0 + \frac{1}{2S} + \frac{1}{b} \big)^2 - \tht_0^2$. Also, from \eqref{e51}, a sufficient condition, holding for all $S$, is $\dsp(\tht-\tht_0)^2\le \frac{2}{b} \tht$ which gives $\theta\in \Big]\tht_0+ \frac{1}{b}- \sqrt{\widetilde{\Delta}},\tht_0+ \frac{1}{b}+ \sqrt{\widetilde{\Delta}}\Big[$ with $\widetilde{\Delta} = \frac{1}{b} (2\tht_0+\frac{1}{b})$, that is  $ p_0\prec p $ if $$\theta\in \Big]\tht_1,\tht_2\Big[:=\Big] \frac{a+1}{b} - \frac{\sqrt{2a+1}}{b},\frac{a+1}{b} + \frac{\sqrt{2a+1}}{b}\Big[.$$

Clearly, one obtains the same result for comparing $\widehat{\theta}_S$ with $\tht_S$. For example, if one chooses  $a=1$, $b= \frac{1}{\tht_0}$ (so that  $\esp(\tau) = \theta_0$) then $\tht_1 = \frac{2-\sqrt{3}}{b}$ and $\tht_2 = \frac{2+\sqrt{3}}{b}$. If $b$ is small, $\theta_2 - \tht_1$ is large but $\tht$ also~!

\smallskip\

Turning to the MAP estimator, one has to compute
$\arg\!\max_{\tht} L(\tht)$  which is equal to $$\arg\!\max_{\tht} e^{-\tht S} \frac{(\tht S)^{N_S}}{N_S!} \frac{b^a}{\Gamma(a)} \tht^{a-1} e^{-\tht b}.$$
We have
$$
\frac{\partial \ln L(\tht)}{\partial\tht} = \frac{\partial}{\partial \tht} \Big( - \tht(S+b) + \frac{N_S+a-1}{\tht}\Big)$$
hence $\widetilde{\tht}_S = \frac{N_S+ a - 1}{b+S}$ where we choose $a\ge 1$ for convenience, inducing the predictor: $$\widetilde{p}_0(N_S) = \frac{N_S+a - 1}{b+S} h + N_S.$$ Replacing $a$ with $a-1$, the previous discussion about $\widehat{p}_0$ holds and one gets, for all $S$, the sufficient condition
$$\widetilde{p}_0 \prec p \Longleftarrow \frac{a}{b} - \frac{{\sqrt{2a-1}}}{b} <\tht <\frac{a}{b} + \frac{{\sqrt{2a-1}}}{b}.$$

\smallskip\

Finally, another method consists in computing the marginal distribution of $(N_S,N_{S+h})$ and then to determine the conditional mode of $N_{S+h}$ given $N_S$. With that method, one obtains a similar predictor. Details are left to the reader.
\subsection{Simulations}
In this section, we compare the unbiased (UP), the Bayesian (BP) and the MAP predictors for various Poisson processes. First, we simulate $N=10^5$ homogeneous Poisson processes with intensity $\theta$ varying in $\{0.5,1,2,5,10\}$. Next, for $S$ in $\{10,15,20,25,30,40,50,75,100\}$ and horizon of prediction $h$ in $\{0.5,1,2,5\}$, we compute an  empirical $L^2$-error of prediction: $$\frac{1}{N} \sum_{j=1}^N \big(N_{S+h}^{(j)} - \widehat{p}(N_{S}^{(j)})\big)^2$$ where $N_{t}^{(j)}$ stands for the $j$-th replicate of the process at time $t$ and $\widehat{p}(N_{S}^{(j)})$ is the predictor under consideration (Bayesian and MAP predictors are computed with a $\Gamma(a,1)$  distribution for the prior). We will also consider the empirical $L^2$-error of estimation (with respect to the probabilistic predictor $\esp_{\tht}(N_{S+h}| N_S)$) defined by $$\frac{1}{N} \sum_{j=1}^N \big(N_S^{(j)} + \tht h - \widehat{p}(N_{S}^{(j)})\big)^2.$$  In Table~\ref{tab1}, we give the rounded $L^2$-errors of estimation according to $S$ as well as prediction errors (enclosed in parentheses) for the unbiased predictor when $\tht= h=1$. To help the comparison, only  the percentage variations of BP and MAP errors (relatively to the UP ones) are reported for $a=1,2,4$. 
Namely, since $\tht=1$, it is expected from \eqref{e51} that $a=4$ represents a bad choice of prior (while $a=1$ corresponds to the best one, and $a=2$ is acceptable). From Table~\ref{tab1}, we observe that:
\begin{itemize}
\item[-] as expected, all errors decrease as $S$ increases~;
\item[-] for all errors and any value of $S$, Bayesian and MAP predictors are better than the unbiased one for  $a=1,2$, with a clearly significant gain for small values of $S$ in the estimation framework~;
\item[-] the bad choice $a=4$ clearly penalizes the predictor, with a significant impact on the $L^2$-error of estimation. Concerning the prediction error, it appears as less sensitive to the prior: indeed this overall error is governed by the probabilistic one, much more important in this case.
\end{itemize}

\begin{table}[h]
\caption{$L^2$ estimation (prediction) error for UP and percentage variation of $L^2$ estimation (prediction) error for BP and MAP, in the case where  $\tht=1$ and $h=1$.}\label{tab1}
\renewcommand{\arraystretch}{1.2}
\scalebox{0.68}{\begin{tabular}{|c||c|c|c||c|c|c||c|c|c|}
\hline
&\multicolumn{3}{|c||}{\bf S=15}&\multicolumn{3}{|c||}{\bf S=20}&\multicolumn{3}{|c|}{\bf S=30}\\
\hline
\bf UP&\multicolumn{3}{|c||}{\bf 0.066 (1.066)}&
\multicolumn{3}{|c||}{\bf 0.050 (1.050)}&
\multicolumn{3}{|c|}{\bf 0.033 (1.036)}
\\\hline
&\bf a=1&\bf a=2&\bf a=4&\bf a=1&\bf a=2&
\bf a=4&\bf a=1&\bf a=2&\bf a=4\\\hline
\bf BP \%&  -12.1(-.74)& -6.3(-.42)& 40.8(2.42)&
-9.3(-.43)&
-4.8(-.23)& 31.5(1.45)& -6.3(-.21)& -3.2(-.11)& 21.8(.69)\\\hline
\bf MAP \%&-6.1(-.33)&-12.1(-.74)&11.3(.64)&-4.7(-.19)&
-9.3(-.43)&8.8(.39)&-3.2(-.10) &-6.3(-.21)&6.2(.19)
\\\hline\hline
&\multicolumn{3}{|c||}{\bf S=40}&\multicolumn{3}{|c||}{\bf S=50} &\multicolumn{3}{|c|}{\bf S=100}\\\hline
\bf UP &\multicolumn{3}{|c||}{\bf 0.025 (1.027)}&\multicolumn{3}{|c||}{\bf 0.020 (1.025)}&\multicolumn{3}{|c|}{\bf 0.010 (1.015)}\\\hline
&\bf a=1&\bf a=2&\bf a=4&\bf a=1&\bf a=2&\bf a=4&\bf a=1&\bf a=2&\bf a=4\\\hline
\bf BP \%&-4.8(-.12)&-2.4(-.05)&16.8(.42)&-3.9(-.08)
&-1.9(-.03)&
13.7(.28)&-2.0(-.02)&-1.0(-.01)&6.7(.06)\\\hline
\bf MAP \%&-2.5(-.06)&-4.8(-.12)&4.8(.13)&-2.0(-.04)&
 -3.9(-.08)&4.0(.09)&-0.9(-.01)&-2.0(-.02)&1.9(.02)\\
 \hline\end{tabular}}\end{table}

\begin{figure}[h]\label{fig1}
\includegraphics[height=6.2cm]{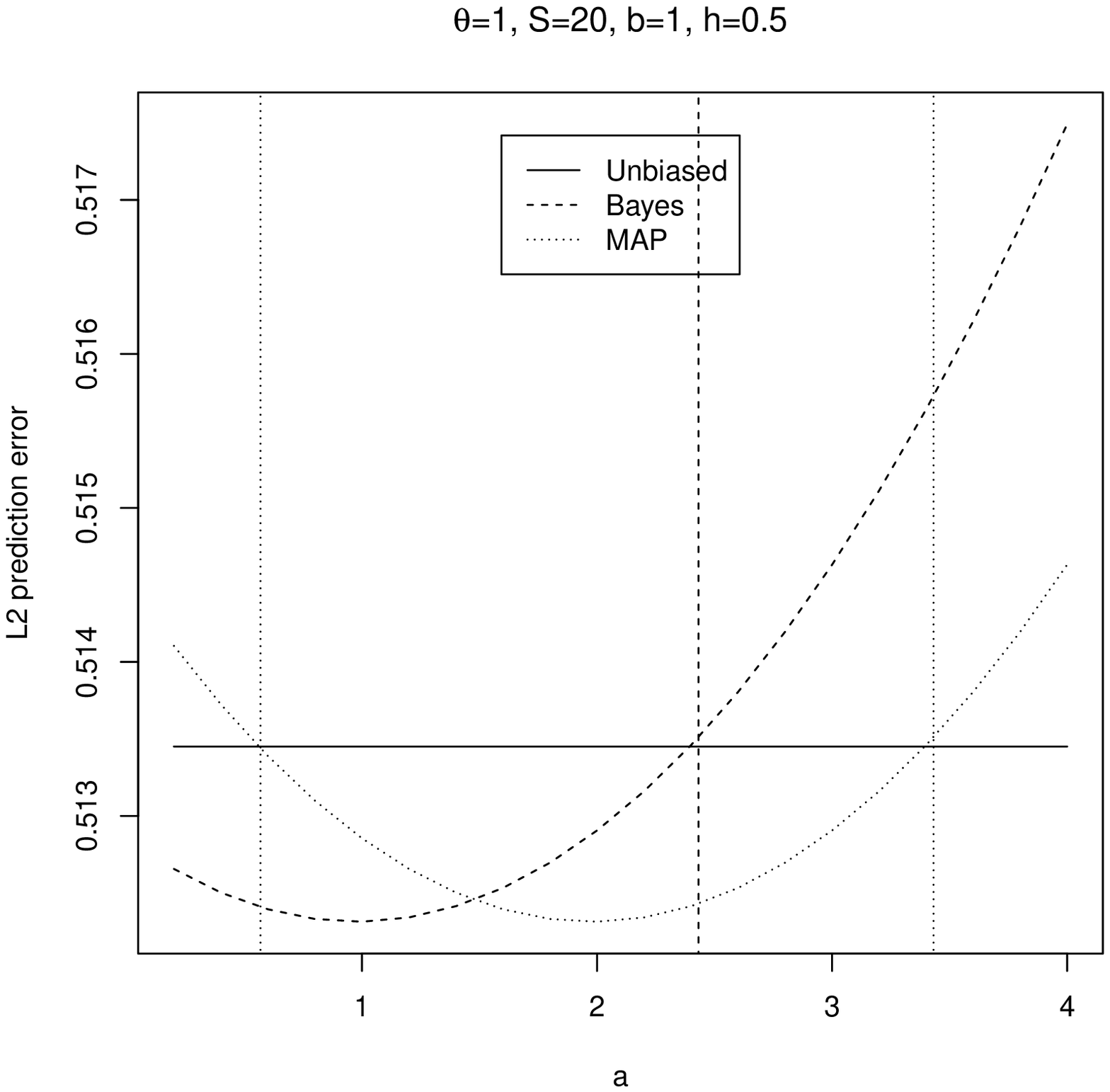}
\includegraphics[height=6.2cm]{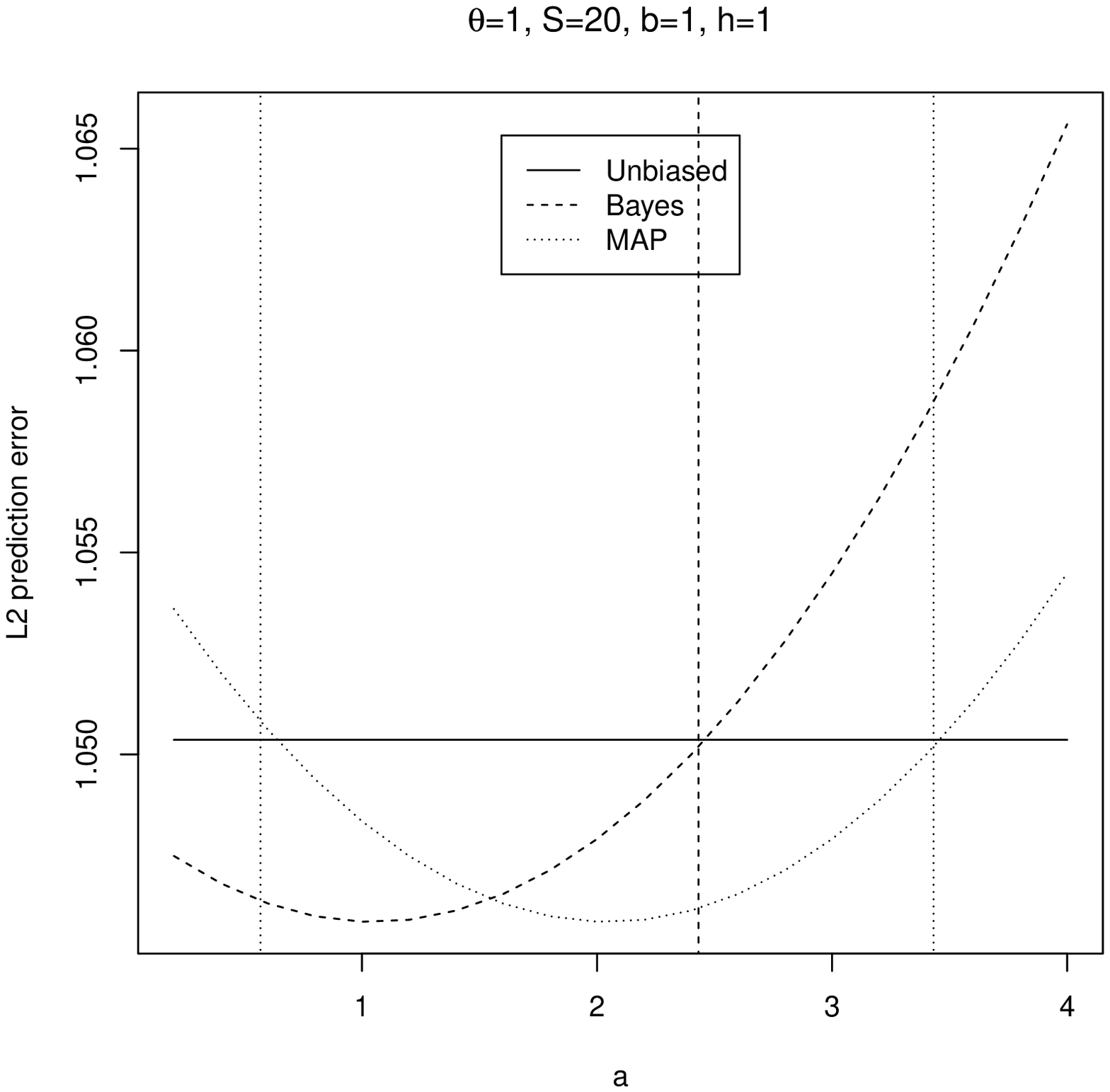}
\caption{{\small$L^2$ prediction error for $\tht=1$ in terms of $a$ with $\Gamma(a,1)$ prior: UP (plain horizontal), BP (dashes), MAP (dots) for $S=20$. Vertical lines corresponds to $a=1 + \sqrt{S^{-1}+2}$ (dashes) and $a=2 \pm \sqrt{S^{-1}+2}$ (dots). On the left : $h=0.5$, on the right : $h=1$.  }}
\end{figure}

In Figure 5.1, the $L^2$-error of prediction is plotted as a function of $a$ for $\tht=1$ and $S=20$. As expected by \eqref{e51}, parabolic curves are obtained and BP (resp. MAP) is better than UP for $a$ in the interval $\Big]\,0, 1 + \sqrt{S^{-1}+2}\,\Big[$ (resp. $\Big]\,2 - \sqrt{S^{-1}+2}, 2+ \sqrt{S^{-1}+2}\,\Big[$). Same conclusions hold for other choices of $h$ and$|$or $\tht$ (see related results of Table~\ref{tab2}). Errors increase as $h$ and|or $\tht$ increase, and a good choice of the prior has a significative impact on the estimation error.


\begin{table}[h]
\caption{$L^2$ estimation (prediction) error, in the case $S=20$, for UP and percentage variations of $L^2$ estimation (prediction) error for BP$i$ and MAP$i$, where $i$ refers to $a=i$.}\label{tab2}
\renewcommand{\arraystretch}{1.2}
\scalebox{0.62}{\begin{tabular}{|c||c|c|c||c|c|c||c|c|c|}
\hline
&\multicolumn{3}{|c||}{\bf $\mathbf\tht$=0.5}&\multicolumn{3}{|c||}{\bf $\mathbf\tht$=5}&\multicolumn{3}{|c|}{\bf $\mathbf\tht$=10}\\
\hline
&\bf h=0.5 &\bf h=1 & \bf h=2& \bf h=0.5 &\bf h=1 & \bf h=2&\bf h=0.5 &\bf h=1 & \bf h=2\\
\hline
\bf UP&\bf .01  (.3) &\bf .02 (.5)  &\bf .1 (1.1) &\bf .06 (2.6) &\bf .25 (5.3) &\bf 1 (11)&\bf .12  (5.1)&\bf .5
(10.4)& \bf 1.99 (22.1)
\\\hline
\bf  BP1 \% &-6.9 (-.16)&-6.9 (-.33)&-6.9 (-.61)&5.1 (.2)&5.1 (.34)&5.1 (.61)&28 (.75)&28 (1.48)&28 (2.7)\\
\hline
\bf BP2 \%&11.5 (.31)&11.5 (.55)&11.5 (1.1)&-1.2 (.02)&-1.2 (.02)&-1.2 (.0)&20.2 (.55)&20.2 (1.09)&20.2 (1.98)\\
\hline
\bf BP4 \%&103.3 (2.58)&103.3 (4.89)&103.3 (9.52)&-8.4 (-.19)&-8.4 (-.37)&-8.4 (-.73)&7.4 (.23)&7.4 (.45)&7.4 (.78)\\
\hline
\bf MAP1 \%&-7.1 (-.19)&-7.1 (-.34)&-7.1 (-.66)&13.2 (.41)&13.2 (.75)&13.2 (1.39)&36.7 (.97) &36.7 (1.91)&36.7 (3.5) \\
\hline
\bf MAP2 \%&-6.9 (-.16)&-6.9 (-.33)&-6.9 (-.61)&5.1 (.2)&5.1 (.34)&5.1 (.61)&28 (.75) &28 (1.48)&28 (2.7)  \\
\hline
\bf MAP4 \%&48.3 (1.22)&48.3 (2.29)&48.3 (4.48)& -5.7 (-.1)&-5.7 (-.22)&-5.7 (-.44)&13.3 (.38) &13.3 (.75)&13.3 (1.34)\\
\hline\end{tabular}}
 \end{table}

\section{Bayesian inference for the Ornstein-Uhlenbeck process} \label{s6}
Consider a stationary version of the Ornstein-Uhlenbeck process (O.U.) defined by
$X_t = m +\int_{-\infty}^t e^{-\tht(t-s)} \di W(s)$, $t\in\reel$, $(m\in\reel, \;\tht>0)$
where $W$ is a standard bilateral Wiener process. Set $X_{0,t} = X_t - m$, $t\in\reel$, then the likelihood of $X_{(S)} = (X_t,\, 0\le t \le S)$ with respect to $X_{0,(S)}=(X_{0,t},\, 0\le t \le S)$ is given by
\begin{equation} \label{e61}
L\big( X_{(S)};m,\tht) = \exp\Big(- \frac{\tht m^2}{2} (2+\tht S) + \tht m (X_0+X_S + \tht \int_0^S X_t \dt)\Big)
\end{equation}
\citep[cf.][p. 128-129]{Gr81}  where $X_{(S)}$ and $X_{0,(S)}$ take their values in the space $C([0,S])$, $(S>0)$.

\subsection{Estimating $m$}

We suppose that $\tht$ is known and $m\in\reel$ is unknown. In order to construct a Bayesian estimator of $m$ and a Bayesian predictor of $X_{S+h}$ ($h>0$) given $X_{(S)}$, we consider the random variable $\mathbb{M}$ with prior distribution ${\cal N}(m_0,u^2)$ ($u>0$), and suppose that $\mathbb{M}$ is independent from $W$. Using \eqref{e61}, it follows that the posterior density of $\mathbb{M}$ given $X_{(S)}$ is $\dsp {\cal N}\big(\frac{B}{A}, \frac{1}{A}\big)$ where $A= \tht (2+\tht S) + \frac{1}{u^2}$ and $B= \tht Z_S + \frac{m_0}{u^2} \;\;\;\text{ with } Z_S =  \big( X_0 + X_S + \tht \int_0^S X_t \dt\big)$. Hence the Bayesian estimator of $m$:
$$
\widehat{m}_S = \frac{B}{A} = \frac{Z_S +m_0 \tht^{-1} u^{-2}}{2+ \tht S + \tht^{-1}u^{-2}}$$
when the maximum likelihood estimator (MLE) is $m_S = \frac{Z_S}{2+\tht S}$. Consequently
\begin{equation}
\label{e62}
\widehat{m}_S= \alpha_S\,m_S+ (1- \alpha_S) m_0
\end{equation}
with $\alpha_S = (1+ \tht^{-1} (2+\tht S)^{-1} u^{-2})^{-1} \in ]0,1[$. Note that $\dsp \lim_{u\to 0} \widehat{m}_S = m_0$ and $\dsp\lim_{u\to\infty} \widehat{m}_S = m_S$.

\paragraph{Asymptotic efficiency}
The MLE $m_S$ is efficient \citep[cf.][p. 28]{BB07}  and $\widehat{m}_S$ is asymptotically efficient since, from \eqref{e62},
$$\frac{\esp_m(\widehat{m}_S - m)^2}{\esp_m (m_S - m)^2} = \alpha_S^2+ (1-\alpha_S)^2 \frac{(m_0-m)^2}{\esp_m(m_S - m)^2}$$
with $\alpha_S^2\to 1$ as $S\to\infty$, $(1-\alpha_S)^2 = {\cal O}(S^{-2})$ and $\esp_m(m_S - m)^2 = {\cal O}(S^{-1})$.

\paragraph{Prediction}
We have
$\esp_m(X_{S+h} | X_{(S)}) = \esp_m(X_{S+h}| X_S) = e^{-\tht h} (X_S- m) + m$.  The unbiased predictor associated with the MLE is
$$p_S := p(X_{(S)}) = m_S ( 1- e^{-\tht h}) + e^{-\tht h} X_S,$$
and by Proposition~\ref{p2}, one obtains the Bayesian predictor
$$\widehat{p}_{0,S}:= p_0(X_{(S)}) = \widehat{m}_S ( 1- e^{-\tht h}) + e^{-\tht h} X_S.$$
We get
$$
\widehat{p}_{0,S} = \alpha_S\, p_S + (1- \alpha_S)\big( m_0(1 - e^{-\tht h})+ e^{-\tht h} X_S\big)= \alpha_S\,  p_S + (1- \alpha_S)p(X_S,m_0).
$$

Concerning efficiency, again we deduce that $p_S$ is efficient and $\widehat{p}_{0,S}$ is asymptotically efficient. Now, in order to compare $\widehat{p}_{0,S}$ with $p_S$, we use Lemma~\ref{l4b} for obtaining the following result.

\begin{proposition}\label{p4b} We have $$\widehat{p}_{0,S} \prec p_S \Longleftrightarrow\abs{m-m_0} \le \Big( \frac{1}{\tht(2+\tht S)} + 2u^2\Big)^{\frac{1}{2}}$$
and $\dsp \abs{m- m_0} \le u\sqrt{2}$  implies  $\widehat{p}_{0,S} \prec p_S $  for all  $S>0$.
\end{proposition}
The proof is straightforward since one has $\esp_m(m_S -m)^2 = \big(\tht (2+\tht S)\big)^{-1}$. Of course, the result is strictly the same if one compares $\widehat{m}_S$ with $m_S$ since $\widehat{m}_S \prec  m_S$ is equivalent to $\widehat{p}_{0,S} \prec p_S$.

\subsection{Estimating $\tht$}

Suppose now that $\tht$ is unknown and $m$ is known~;~one may take $m=0$. The likelihood of $X_{(S)}$ with respect to $W_{(S)}$ has the form
\begin{equation*}
L(X_{(S)}) = \exp\Big(-\frac{1}{2} (X_S^2- X_0^2 -S) - \frac{\tht^2}{2}\int_0^S X^2_t \dt\Big),
\end{equation*}
see \citet{LS01}. Even if $\tht$ is positive, it is convenient to take ${\cal N}(\tht_0,v^2)$ (with $\tht_0 >0$ and $v^2>0$) as prior distribution of $\mathbb{T}$. Then, the marginal distribution of $X_{(S)}$ has density
$\vp(x_{(S)}) =\frac{1}{\sqrt{\alpha v^2}} \exp\Big( - \frac{\tht_0^2}{2v^2} + \frac{\beta^2}{2\alpha}\Big)$
where $\alpha =\int_0^S x_s^2 \ds+ \frac{1}{v^2}$ and $\beta = \frac{S- x_S^2+x_0^2}{2} + \frac{\tht_0}{v^2}$.

It follows that the conditional distribution of $\mathbb{T}$ given $X_{(S)}$ is ${\cal N}\big( \frac{\beta}{\alpha}, \frac{1}{\alpha}\big)$, hence the Bayesian estimator of $\tht$:
$
\widehat{\tht}_S = \frac{\beta}{\alpha}= \frac{\frac{1}{2}(S- X_S^2+X_0^2)+ \tht_0 v^{-2}}{\int_0^S X_t^2 \dt + v^{-2}}$
when the MLE is $\theta_S = \frac{\frac{1}{2}(S - X_S^2+X_0^2)}{\int_0^S X_t^2 \dt}$,
consequently
\begin{equation}\label{e63}
\widehat{\tht}_S = \gm_S\, \tht_S + (1-\gm_S)\tht_0 \;\;\text{ with }\;\; \gm_S = \frac{\int_0^S X_t^2\dt}{\int_0^S X_t^2 \dt +v^{-2}},
\end{equation}
and $\dsp \lim_{v^2\to 0} \widehat{\tht}_S= \tht_0$ while $\dsp \lim_{v^2\to\infty} \widehat{\tht}_S = \tht_S$.

Concerning prediction, we have $\dsp \esp_{\tht} (X_{S+h} |X_{(S)}) = e^{-\tht h} \cdot X_S$, so it is necessary to compute the Bayesian estimator of $e^{-\tht\,h}$. We get
$$
\esp(e^{-\mathbb{T} h} | X_{(S)} ) =\int_{\reel} e^{-\tht h} \sqrt{\frac{\alpha}{2\pi}} e^{-\frac{\alpha}{2}(\tht -\frac{\beta}{\alpha})^2} \di\tht
= \exp(- \frac{2\beta -h}{2\alpha}\cdot h),
$$
hence the Bayesian predictor $\widehat{p}_0(X_{(S)}) = \exp(-\frac{2\beta - h}{2\alpha} \cdot h)\cdot  X_S$. The predictor associated with the MLE is $p(X_{(S)}) =e^{-\tht_S \cdot h} \cdot X_S$
and finally, an alternative form of the predictor, associated with the MAP, should be
$\widetilde{p}(X_{(S)}) = e^{-\widehat{\tht}_S\cdot h} \cdot X_S$.

\medskip\


Finally, one may consider alternative priors, as well as, the translated exponential distribution with density $\vp(\tht) = \eta \exp\big(-\eta (\tht - \tht_0)\big) \indi_{]\tht_0,+\infty[}(\tht)$, $(\eta>0,\, \tht_0 \ge 0)$. If $\psi$ denotes the density of ${\cal N}(- \frac{a}{2b}, \frac{1}{b})$, with $a = x_S^2 - x_0^2 - S + 2\eta$ and $b= \int_0^S x_t^2\dt$, the Bayesian estimator is  given by
$\widehat{\tht}_S = \int_{\tht_0}^{\infty} \tht \psi(\tht) \di\tht \Big/ \int_{\tht_0}^{\infty} \psi(\tht) \di\tht$ and can be numerically computed.
Derivation is left to the reader.

\section{Ornstein-Uhlenbeck process for sampled data} \label{s7}

We now consider the more realistic case where only $X_0,X_{\de},\dotsc,X_{n\de}$ are observed and one wants to predict $X_{(n+h)\de}$, $(h>0)$.

\subsection{Estimation of $m$}

If $\tht$ is known, and $m\in \reel$ unknown, the associated model is
\begin{align}\label{e71}
X_{n\de} - m &= e^{-\tht\de} \big( X_{(n-1)\de} - m) + \ve_{n\de}, \;\;\; n \in\z
\intertext{and}
\var(\ve_{n\de}) &= \frac{1- e^{-2\tht \de}}{2\tht} =: \sigma^2_{\de,\tht} \label{e72}
\end{align}
If $\de>0$ is fixed, we deal with a classical AR(1), so we will focus on the case  where $\de=\dn$ is `small'. One may use various conditions as $n\to \infty$: $\dn \to 0$ and $n\dn \to \infty$ or $\dn \to 0$ and $n\dn \to S >0$ for example. Two approaches are possible: either considering the likelihood or the conditional likelihood  ($X_0$ is arbitrary but non random) which has a simpler form.

\subsubsection{Unconditional estimation}\label{sub711}
Since $X_0-m,\ve_{\dn},\dotsc,\ve_{n\dn}\sim {\cal N}(0,(2\tht)^{-1}) \otimes {\cal N}(0, \sigma^2_{\delta_n,\tht})^{\otimes n}$, one may deduce that  $(X_0-m,X_{\delta_n}-m,\dotsc,X_{n\delta_n}-m)$ has the density
\begin{multline*}
f(x_0,x_1,\dotsc,x_n) = \Big(\frac{\tht}{\pi}\Big)^{\frac{1}{2}}  \frac{1}{(\sigma_{\delta_n,\tht}\sqrt{2\pi})^n} \times \exp\Big(-\tht(x_0-m)^2\\-\sum_{i=1}^n \frac{\big(x_i - e^{-\tht\dn} x_{i-1}-  m(1-e^{-\tht\dn})\big)^2}{2\sigma_{\delta_n\,\tht}^2}\Big).
\end{multline*}
This yields
\begin{equation}\label{e73}
m_n = \frac{X_0 + X_{n\dn} + (1- e^{-\tht \dn}) \sum_{i=1}^{n-1} X_{i\dn}}{n(1- e^{-\tht \dn}) + 1+ e^{-\tht \dn}}
\end{equation}
for the MLE, while if $\mathbb{M}\sim {\cal N}(m_0,u^2)$, one has
\begin{multline*}
L(X_0,X_{\dn},\dotsc,X_{n\dn},\mathbb{M})=  \Big(\frac{\tht}{\pi}\Big)^{\frac{1}{2}}  \frac{1}{\sigma_{\delta_n,\tht}\sqrt{2\pi}} \times \exp\Big(-\tht(X_0-\mathbb{M})^2\\-\sum_{i=1}^n \frac{\big(X_{i\dn} - e^{-\tht\dn} X_{(i-1)\dn}-  \mathbb{M}(1-e^{-\tht\dn})\big)^2}{2\sigma_{\delta_n\,\tht}^2}\Big)\times \frac{1}{u\sqrt{2\pi}}\exp\big(-\frac{1}{2u^2}(\mathbb{M}-m_0)^2\big)
\end{multline*}
giving
\begin{equation}\label{e74}
\widehat{m}_n = \frac{X_0 + X_{n\dn} + (1- e^{-\tht \dn}) \sum_{i=1}^{n-1} X_{i\dn}+(1+e^{-\tht\dn})\frac{m_0}{2\tht u^2}}{n(1- e^{-\tht \dn}) + (1+ e^{-\tht \dn})(1+ \frac{1}{2\tht u^2})}.
\end{equation}
Again, we have $\widehat{m}_n =\alpha_n m_n + (1-\alpha_n) m_0$ with
$$\alpha_n = \frac{n(1-e^{-\tht\dn}) + 1+ e^{-\tht \dn}}{n(1-e^{-\tht\dn})  + (1+ e^{-\tht \dn})(1+ (2\tht u^2)^{-1})}.$$
Since $
\esp(X_{(n+h)\dn} | X_{n\dn}) = e^{-\tht h \dn} (X_{n\dn} - m) +m$, the derived predictors of $X_{(n+h)\dn}$, $h \ge 1$ are  given by $p_n(X_{n\dn})= m_n(1-e^{-\tht h\dn}) + e^{-\tht h \dn} X_{n\dn}$ while $\widehat{p}_{0,n}(X_{n\dn}) =  \widehat{m}_n(1-e^{-\tht h\dn}) + e^{-\tht h \dn} X_{n\dn}$,  and Lemma~\ref{l4b} implies that
\begin{multline*}
\widehat{p}_{0,n} \prec p_n \eqv\\ (m- m_0)^2\le  \frac{2n(1-e^{-\tht\dn}) + (1+ e^{-\tht \dn})(2+ (2\tht u^2)^{-1})}{ (1+ e^{-\tht \dn})(2\tht u^2)^{-1}}\times \esp_m(m_n -m)^2.
\end{multline*}
Next, easy but tedious computation gives $\esp_m(m_n - m)^2 =\frac{1+e^{-\tht \dn}}{2\tht \big(n(1-e^{-\tht \dn})+ 1+e^{-\tht \dn}\big)}$ yielding the equivalence: $\widehat{p}_{0,n} \prec p_n \Leftrightarrow (m-m_0)^2\le 2 u^2 + \frac{1+e^{-\tht \dn}}{2\tht \big( 1+e^{-\tht\dn}+n(1- e^{-\tht \dn})\big)}$. Asymptotically, we get if $\dsp \dn\tv[n\to\infty]{} 0$, $\dsp n\dn \tv[n\to\infty]{} S>0$, $\widehat{p}_{0,n}\prec_{n\to\infty} p_n $ is equivalent to  $(m-m_0)^2\le 2u^2+ \frac{1}{\tht(\tht S+2)}$. The condition $S\to\infty$ implying in turn the equivalence $
\widehat{p}_{0,n} \prec_{n\to\infty} p_n \Leftrightarrow (m-m_0)^2\le 2u^2$, which are the same results as in the continuous case (cf. Proposition~\ref{p4b}). If $n\dn\to S>0$, note that our estimators of $m$ are no more consistent~! But still in this case, a good choice of the prior should allow reductions  of risks of estimation and prediction.
\subsubsection{Conditional likelihood}
In this part, we use conditional likelihood on $X_0$, and choosing $\mathbb{M}\sim {\cal N}(m_0,u^2)$, $(u>0)$, we obtain the `density' of $(X_{\dn},\dotsc,X_{n\dn},\mathbb{M})$:
\begin{multline*}
\widetilde{L}(X_{\dn},\dotsc,X_{n\dn},\mathbb{M}) = \frac{1}{( \sigma_{\dn,\tht} \sqrt{2\pi})^n} \exp\bigg( - \frac{1}{2 \sigma^2_{\delta_n,\tht}} \sum_{i=1}^n \Big( (X_{i\dn} - \exp(-\tht \dn) X_{(i-1)\dn}) \\+ \mathbb{M}(\exp(-\tht \dn)-1)\Big)^2\bigg)  \times \frac{1}{u\sqrt{2\pi}} \exp\Big( - \frac{1}{2u^2} (\mathbb{M}- m_0)^2\Big),
\end{multline*}
where $\sigma^2_{\delta_n,\tht}$ is defined by \eqref{e72}. Now:
$$\ln\widetilde{L} = c - \frac{1}{2\sigma^2_{\delta_n,\tht}} \sum_{i=1}^n \Big( X_{i\dn} - \exp(-\tht \dn) X_{(i-1)\dn}+\mathbb{M}(\exp(-\tht \dn) -1)\Big)^2 - \frac{(\mathbb{M}-m_0)^2}{2u^2},$$
where $c$ does not depend on $n$. Since we are in the Gaussian case, the conditional mode and the conditional expectation coincide and it follows that the Bayesian estimator is now given by
\begin{equation}\label{e75}
\widetilde{m}_n = \frac{(1-\exp(-\tht \dn)) \sum_{i=1}^n (X_{i\dn} - \exp(-\tht \dn) X_{(i-1)\de}) + m_0 \frac{\sigma^2_{\delta_n,\tht}}{u^2}}{(1-\exp(-\tht \dn))^2 n + \frac{\sigma^2_{\delta_n,\tht}}{u^2}},
\end{equation}
while the conditional MLE takes the form
\begin{equation}\label{e76}
\breve{m}_n = \frac{\sum_{i=1}^n (X_{i\dn} - \exp(-\tht \dn) X_{(i-1)\de}) }{(1-\exp(-\tht \dn)) n },
\end{equation}
We may slightly modify the estimator \eqref{e75} for obtaining
\begin{equation} \label{e77}
\overline{m}_n = \beta_n \overline{X}_n + (1-\beta_n) m_0
\end{equation}
with $\overline{X}_n = n^{-1}\sum_{i=1}^n X_{i\dn}$ and $\dsp
\beta_n = \frac{(1-\exp(-\tht \dn))^2}{(1-\exp(-\tht \dn))^2+ \frac{\sigma^2_{\delta_n,\tht}}{nu^2}}$. Hence
$$\frac{1+\beta_n}{1-\beta_n} =\frac{4\theta u^2 n (1-e^{-\tht \dn})+1+e^{-\tht\dn}}{1+e^{-\tht \dn}}$$
and, since $\var(\overline{X}_n)= \frac{(1-e^{-2\tht\dn})+ \frac{2}{n}e^{-\tht\dn}(e^{-\tht n\dn }-1)}{2n\tht(1-e^{-\tht\dn})^2}$,  asymptotically we get that, if  $\dn\to 0,n\dn\to S>0$,
$$
\overline{m}_n \prec_{n\to\infty} \overline{X}_n \eqv (m-m_0)^2 \le \frac{(1+2u^2\tht^2S)(\tht S - 1+e^{-\tht S})}{\tht^3 S^2}$$
while if  $\dn\to 0,n\dn\to \infty$, we get the equivalence: $\overline{m}_n \prec_{n\to \infty} \overline{X}_n \Leftrightarrow (m-m_0)^2 \le 2u^2$. Again, the same results are obtained for predictors.
\subsection{Estimation of $\rho$}
In the case where $m$ is known (one may set $m=0$), we now choose ${N}(\rho_0,v^2)$ as a prior for $\rho = e^{-\tht \dn}$, with $0<\rho_0<1$ and $v>0$. Note that this prior is reasonable as soon as $\rho_0$ is not too far from 1 and $v$  not too large. Using again the conditional likelihood, one obtains the expression:
\begin{multline*}
\widetilde{L}(X_{\dn},\dotsc,X_{n\dn},\rho) =\frac{1}{(\sigma_{\dn,\tht}\sqrt{2\pi})^n}\exp\Big(-\frac{1}{2\sigma^2_{\delta_n,\tht} } \sum_{i=1}^n \big( X_{i\dn} - \rho X_{(i-1)\dn}\big)^2 \Big) \\ \times \frac{1}{v\sqrt{2\pi}} \exp\big( - \frac{1}{2v^2} (\rho - \rho_0)^2\big).
\end{multline*}
Since $\sigma_{\tht,\de}^2$ depends on $\rho$, we make the approximation $\sigma_{\de,\tht}\sim \de$ for obtaining the posterior distribution ${\cal N} (\frac{B}{A},\frac{1}{A})$ where $\dsp A= \frac{1}{\dn} \sum_{i=1}^n X_{(i-1)\dn}^2 + \frac{1}{v^2}$ and $\dsp B= \frac{1}{\dn} \sum_{i=1}^nX_{(i-1)\dn}X_{i\dn}+ \frac{\rho_0}{v^2}$, hence the `Bayesian' estimator takes the form
\begin{equation}\label{e78} \widetilde{\rho}_n = \frac{\sum_{i=1}^n X_{(i-1)\dn}X_{i\dn} + \frac{\rho_0 \dn}{v^2}}{\sum_{i=1}^n X_{(i-1)\dn}^2 + \frac{\dn}{v^2}}.
\end{equation}
Comparison with the conditional MLE  \begin{equation}\label{e79}
\widehat{\rho}_n = \frac{\sum_{i=1}^n X_{(i-1)\dn}X_{i\dn} }{\sum_{i=1}^n X_{(i-1)\dn}^2}
\end{equation}
is rather intricate and will be illustrated numerically in the next section.

\subsection{Simulation}

\begin{table}[b]
\caption{$L^2$-prediction error  ($m$ unknown) for MLE predictor and percentage variation of $L^2$-prediction error for others in the case where  $\tht=1$, $H=1$, $u^2=1$ and $\de=0.1$.}\label{tab3}
\renewcommand{\arraystretch}{1.2}
\scalebox{0.95}{\begin{tabular}{|c||c|c|c||c|c|c||}
\hline
&\multicolumn{3}{|c||}{\bf n=15}&\multicolumn{3}{|c||}{\bf n=30}\\
\hline
\bf MLE &\multicolumn{3}{|c||}{\bf 0.548}&\multicolumn{3}{|c||}{\bf 0.499}\\
\hline
\bf Mean (\%) &\multicolumn{3}{|c||}{2.76}&\multicolumn{3}{|c||}{2.82}\\ \hline
\bf CMLE  (\%)  &\multicolumn{3}{|c||}{27.97}&\multicolumn{3}{|c||}{10.54} \\ \hline
&$\mathbf{m_0=4}$&$\mathbf{m_0=5}$&$\mathbf{m_0=7}$&$\mathbf{m_0=4}$&$\mathbf{m_0=5}$&
$\mathbf{m_0=7}$ \\ \hline
\bf Bay  (\%)  & -4.87 &-8.52 &5.77& -2.79& -4.72 &4.78 \\ \hline
\bf CMAP2  (\%)  &-1.04 &-12.86 &33.54&-.53 &-5.02&16.05 \\ \hline\hline
&\multicolumn{3}{|c||}{\bf n=50}&\multicolumn{3}{|c||}{\bf n=100}\\
\hline
\bf MLE &\multicolumn{3}{|c||}{\bf 0.488}&\multicolumn{3}{|c||}{\bf 0.464}
\\\hline
\bf Mean (\%) &\multicolumn{3}{|c||}{1.61}&\multicolumn{3}{|c||}{0.35}
\\\hline
\bf CMLE  (\%) & \multicolumn{3}{|c||}{5.62}&\multicolumn{3}{|c||}{1.35}
\\\hline
&$\mathbf{m_0=4}$&$\mathbf{m_0=5}$&$\mathbf{m_0=7}$&$\mathbf{m_0=4}$&$\mathbf{m_0=5}$&$\mathbf{m_0=7}$\\\hline
\bf Bay  (\%)  &-1.30 & -2.63 & 2.39&-0.62 &-1.08 &1.08\\
\hline
\bf CMAP2  (\%) & 0.03& -2.28  &6.77& -0.35&  -0.99  &2.00  \\\hline
\end{tabular}}\end{table}

For $\tht\in\{0.5,1,2\}$, $m=5$, various  sample sizes $n$ and values of $\de$, 5000 replications of Ornstein-Uhlenbeck sample paths  are computed from the autoregressive relation \eqref{e71}. First, for known $\tht$ but $m$  unknown,  we compare various predictors of $X_{n\de + H}$, $H= h\de$ and $H=0.5$, 1 or 2, defined by $\mathfrak{m} (1- e^{-\tht h\de})+e^{-\tht h \de} X_{n\de}$ where $\mathfrak{m}$ refers to estimators which are either:
 \begin{itemize}
 \item non Bayesian: MLE with $m_n$ defined in \eqref{e73}, Mean $\overline{X}_n$, CMLE  with $\breve{m}_n$ defined in \eqref{e76},
 \item or Bayesian: Bayes with $\widehat{m}_n$ defined in \eqref{e74}, CMAP1 with $\widetilde{m}_n$ defined in \eqref{e75} ($u^2=1$) and CMAP2 with $\overline{m}_n$ defined in \eqref{e77} ($u^2=1$).
 \end{itemize}

\begin{figure}[t]\label{fig2}
\includegraphics[height=6.2cm]{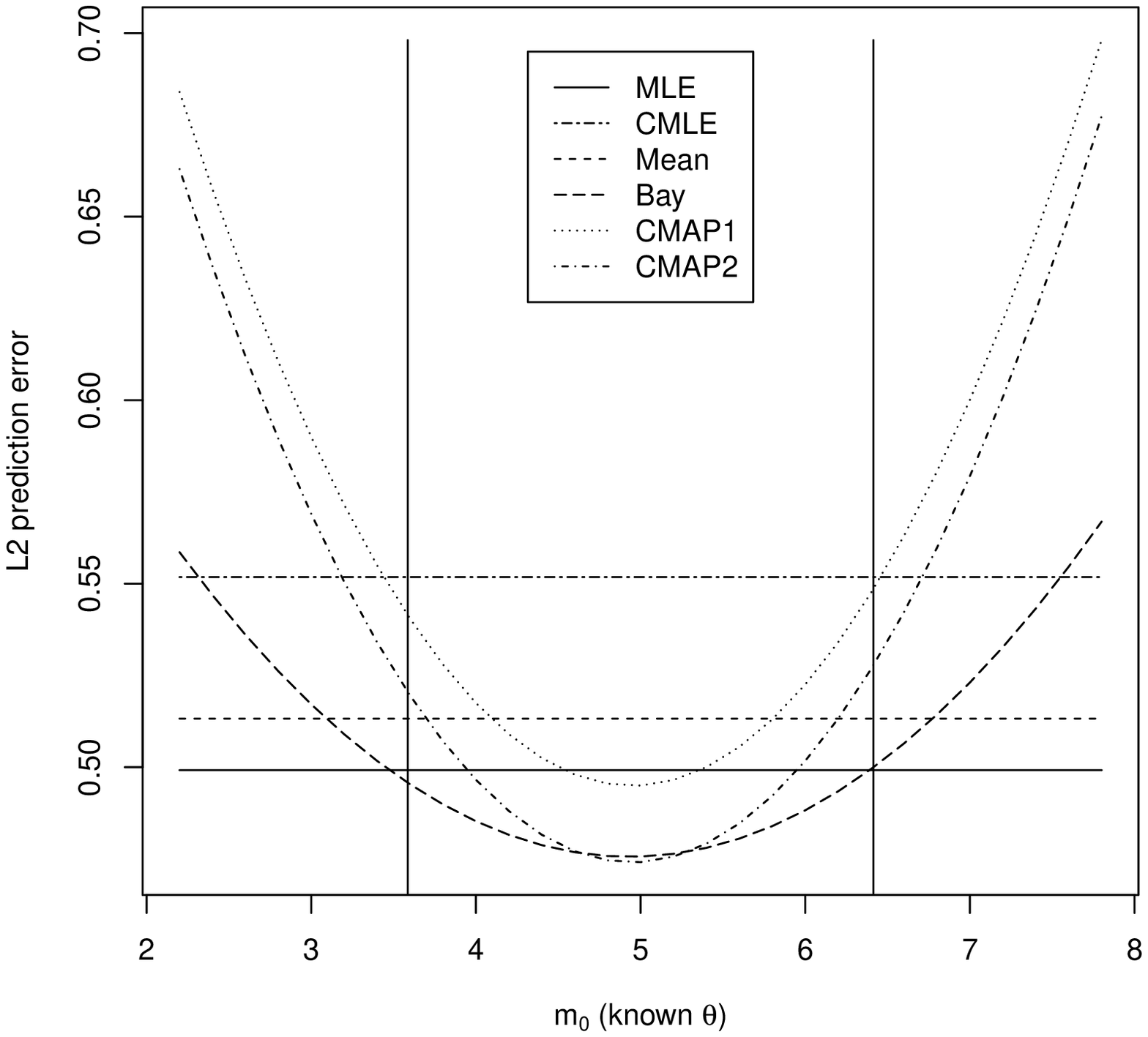}
\includegraphics[height=6.2cm]{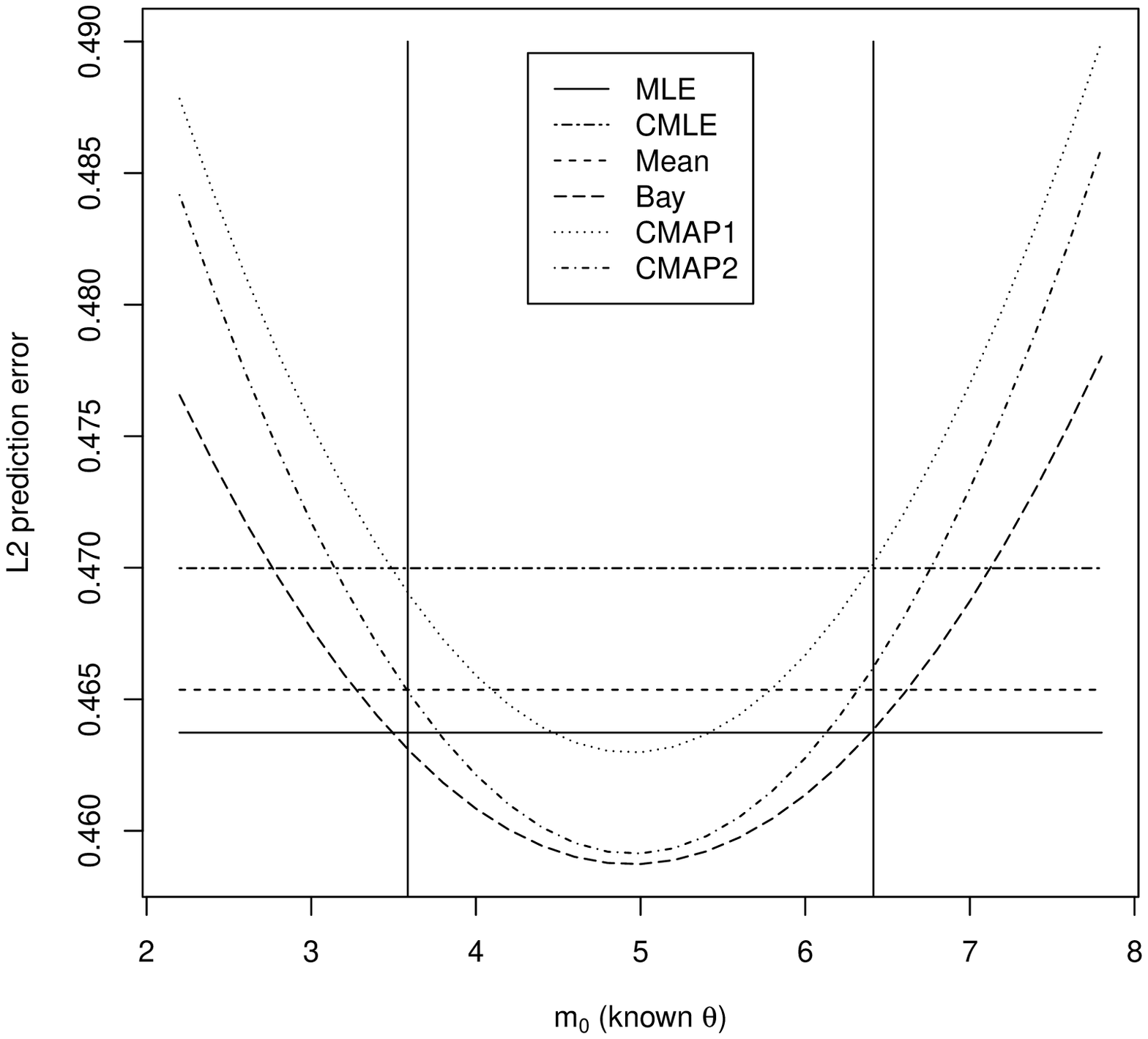}
\caption{{\small $L^2$-prediction error for $m$ unknown ($m=5$) and $\tht=1$ (known), $\de=0.1$  in terms of $m_0$ with ${\cal N}(m_0,1)$ prior: MLE (plain), CMLE (twodash), Mean (dashed), Bayes (longdash), CMAP1 (dotted), CMAP2 (dotdash) when $u^2=1$. Vertical lines corresponds to $m_0= 5\pm \sqrt{2u^2}$. On the left : $n=30 \;(S=3,S+h*\de=4)$, on the right: $n=100 \;(S=10,S+h*\de=11)$. } }
\end{figure}

\begin{table}[b]
\caption{$L^2$-prediction error ($m$ unknown) for MLE predictor and percentage variation of $L^2$-prediction error for others in the case where  $\tht=1$,  $H=1$, $u^2=1$ and $m_0\in\{4,5,7\}$.}\label{tab4}
\renewcommand{\arraystretch}{1.2}
\scalebox{0.72}{\begin{tabular}{|c||c|c|c||c|c|c||c|c|c||c|c|c|}
\hline
\bf n &\multicolumn{3}{|c||}{\bf 10}&\multicolumn{3}{|c||}{\bf 20}&\multicolumn{3}{|c||}{\bf 50}&\multicolumn{3}{|c|}{\bf 100}\\
\hline
$\mathbf{\de}$ & \bf 0.1 & \bf 0.2 & \bf 0.5 & \bf 0.1 & \bf 0.2 & \bf 0.5 & \bf 0.1 & \bf 0.2 & \bf 0.5 & \bf 0.1 & \bf 0.2 & \bf 0.5\\
\hline\hline
\bf MLE& .586 &\bf .531& \bf .488 &\bf .531 & .498 &\bf .464& \bf.488& \bf  .464 &.441&\bf  .464 &.461& .444
\\ \hline
\bf Mean (\%) &2.77& 2.46 & 1.92& 2.40& 2.13 & .44& 1.61 &.36 & .16& .35 &.16 &-.06
\\ \hline
\bf CMLE (\%)& 47.01 & 17.69 & 5.73& 17.56&  7.38 & 1.40&  5.62&  1.36 & 0.24& 1.35 & .47 &-.05
\\ \hline
\bf Bay\textunderscore 4 (\%)&  -5.06& -4.23& -1.31& -4.22 &-1.84& -.65& -1.30 &-.62& -.03&-.62 &-.31 &-.04
\\ \hline
\bf Bay\textunderscore 5 (\%)& -10.48& -6.66 &-2.68& -6.65 &-3.41& -1.11&  -2.63 &-1.08& -0.23 & -1.08 &-0.37& -0.05
\\  \hline
\bf Bay\textunderscore 7(\%)&  4.26 &6.58 &2.47& 6.57& 3.32& 1.11& 2.39& 1.08& 0.11& 1.08 &0.49 &0.12
\\  \hline
\bf CMAP2\textunderscore 4 (\%)&  1.89& -1.70&  .26& -1.70 &-.13& -.30 &.03 &-.34 & .13& -.35 &-.18 &-.10
\\\hline
\bf CMAP2\textunderscore 5 (\%)& -17.35 &-9.23 &-2.13&  -9.23 &-3.20& -.96&  -2.28 &-.98 &-.09& -.99 &-.26 &-.12
\\ \hline
\bf CMAP2\textunderscore 7 (\%)&  46.63& 26.12& 7.21 &25.96 &10.04 &2.15 & 6.77 & 2.02 &.31 & 2.00 & .75 &.07
\\
\hline
\end{tabular}}\end{table}

Among all non Bayesian estimators and in all cases, it emerges that MLE outperforms the other two, with a very poor behaviour of the CMLE toward the others, a fact already noticed by \citet{Co91}. For this reason, our following results do not report the obtained values for CMAP1, because of its too bad behaviour governed by the CMLE. In Table~\ref{tab3}, we give the rounded empirical $L^2$-prediction error of the MLE, and for comparison, the percentage variations observed for the others predictors, in the case of $\tht=1$ and $\de=0.1$. It appears that all errors decrease as $n$ increases, and Bayes predictors are highly competitive for small sample sizes and good choice of priors, namely $\mathbb{M}\sim{\cal N}(m_0,1)$, with $m_0\in \Big] 5- \sqrt{2 + (S+2)^{-1}}, 5+ \sqrt{2 + (S+2)^{-1}}\Big[$ or asymptotically, $S=n\de \to\infty$,  $m_0\in \Big] 5- \sqrt{2}, 5+ \sqrt{2}\Big[$, see Section~\ref{sub711}. By this way, errors are significantly reduced for $m_0=4$ or $5$ and $n$ less than 50, while a bad choice like $m_0=7$ damages them dramatically. It appears also that CMAP2 has the smallest errors but only on a small area around $m$, the Bayesian predictor (with $\widehat{m}_n$ defined in \eqref{e74}) being more robust against the choice of $m_0$. These results are confirmed in Figure~7.1 where errors are given in term of $m_0$: as expected, we obtain parabolic curves for Bayesian predictors. Again, the Bayesian setting improves the errors for good choices of prior (especially for small values of $\de$ and $n\de$ where MLE is not so good) and otherwise deteriorates it.

\begin{figure}[t]\label{fig3}
\includegraphics[height=6.2cm]{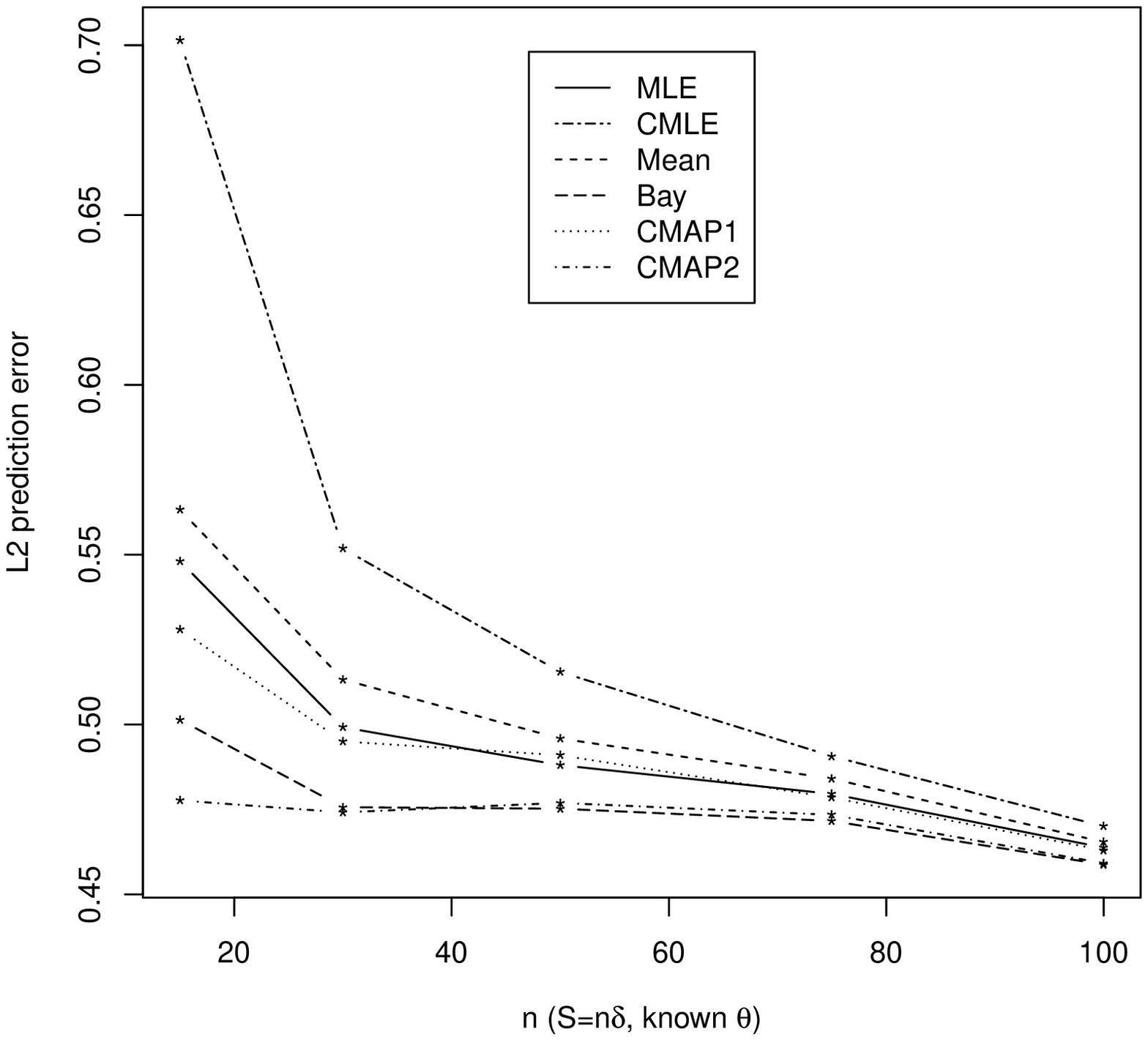}
\includegraphics[height=6.2cm]{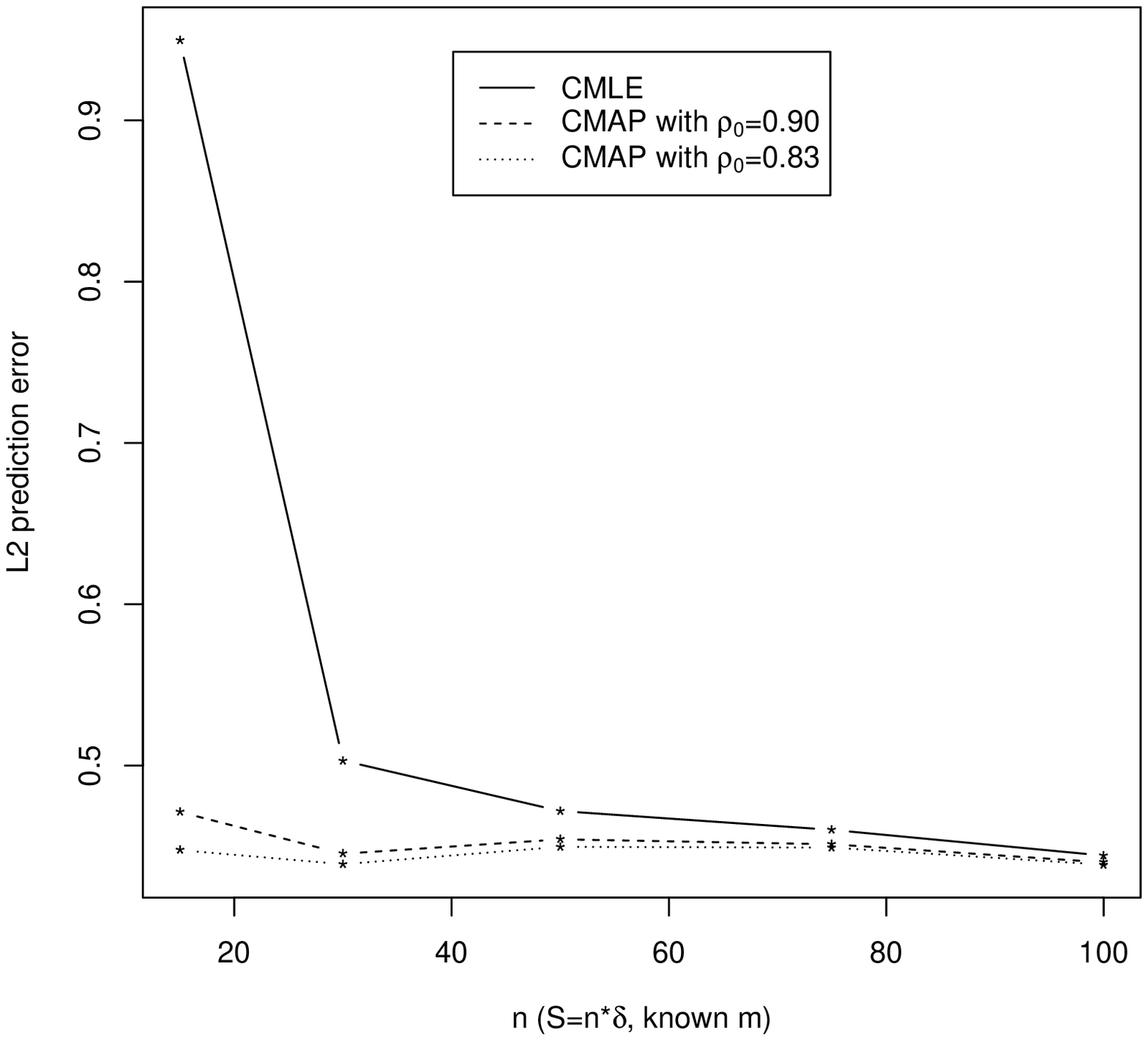}
\caption{{\small On the left: $L^2$ prediction error for unknown $m$, with prior ${\cal N}(5,1)$, known $\tht$ ($\tht=1$) in terms of $n$ when $\de=0.1$:MLE (plain), CMLE (twodash), Mean (dashed), Bayes (longdash), CMAP1 (dotted), CMAP2 (dotdash). On the right: $L^2$ prediction error for unknown $\rho$, prior ${\cal N}(\rho_0,10^{-2})$ in terms of $n$ when $\de =0.1$: CMLE (plain), CBayes with $\rho_0=0.9$ (dashed), CBayes (dotted) with $\rho_0=0.83$.}}
\end{figure}

\begin{table}[h]
\caption{$L^2$-prediction error ($m$ unknown) for MLE predictor and percentage variation of $L^2$-prediction error for other in the case where  $n=20$, $\de=0.1$, $u^2=1$, and $m_0\in\{4,5,7\}$.}\label{tab5}
\renewcommand{\arraystretch}{1.2}
\scalebox{0.82}{\begin{tabular}{|c||c|c|c||c|c|c||c|c|c|}
\hline
$\mathbf \tht$  &\multicolumn{3}{|c||}{$\mathbf{0.5}$}&\multicolumn{3}{|c||}{$\mathbf{1}$}&\multicolumn{3}{|c|}{$\mathbf{2}$}\\
\hline
\bf H & \bf 0.5 & \bf 1 & \bf 2 & \bf 0.5 & \bf 1 & \bf 2 & \bf 0.5 & \bf 1 & \bf 2 \\
\hline\hline
\bf MLE&  \bf 0.421 & \bf 0.728 &\bf  1.138& \bf .34 & \bf .531 &\bf .677 & \bf  .249  & \bf .303  &\bf  .333
\\ \hline
\bf Mean (\%) & 0.51 & 1.32 & 2.23 & 1.81 & 2.4 & 4.07 & 1.88 & 2.72 & 2.85
\\ \hline
\bf CMLE (\%)&   15.03 &27.59& 48.31& 10.76& 17.56 & 27.37&  6.04 &9.76& 10.65
\\ \hline
\bf Bay\textunderscore 4 (\%)&  -3.41& -5.9 &-9.32& -2.72 & -4.22& -6.09 &  -1.04 &-0.92& -2.15
\\ \hline
\bf Bay\textunderscore 5 (\%)&  -5.27 &-9.31 & -15.56 & -4.11 &-6.65 &-9.82&-1.99 &-2.97 &-3.58
\\  \hline
\bf Bay\textunderscore 7(\%)&  2.18& 4.31 &5.68&4.04&6.57&9.24&1.82 & 1.7 &3.87
\\  \hline
\bf CMAP2\textunderscore 4 (\%)&-2.2 &-3.33& -4.77&-1.07 & -1.7&-1.96&0.67 &1.85& 0.33
\\\hline
\bf CMAP2\textunderscore 5 (\%)& -7.41 &-12.86& -21.55&-5.5&-9.23&-13.28& -1.35& -2.1 &-2.9
\\ \hline
\bf CMAP2\textunderscore 7 (\%)&  13.2& 24.79 &38.55&16.06&25.96&37.82& 6.56& 8.36 &12.21
\\
\hline
\end{tabular}}\end{table}

\begin{figure}[t]\label{fig4}
\includegraphics[height=6.2cm]{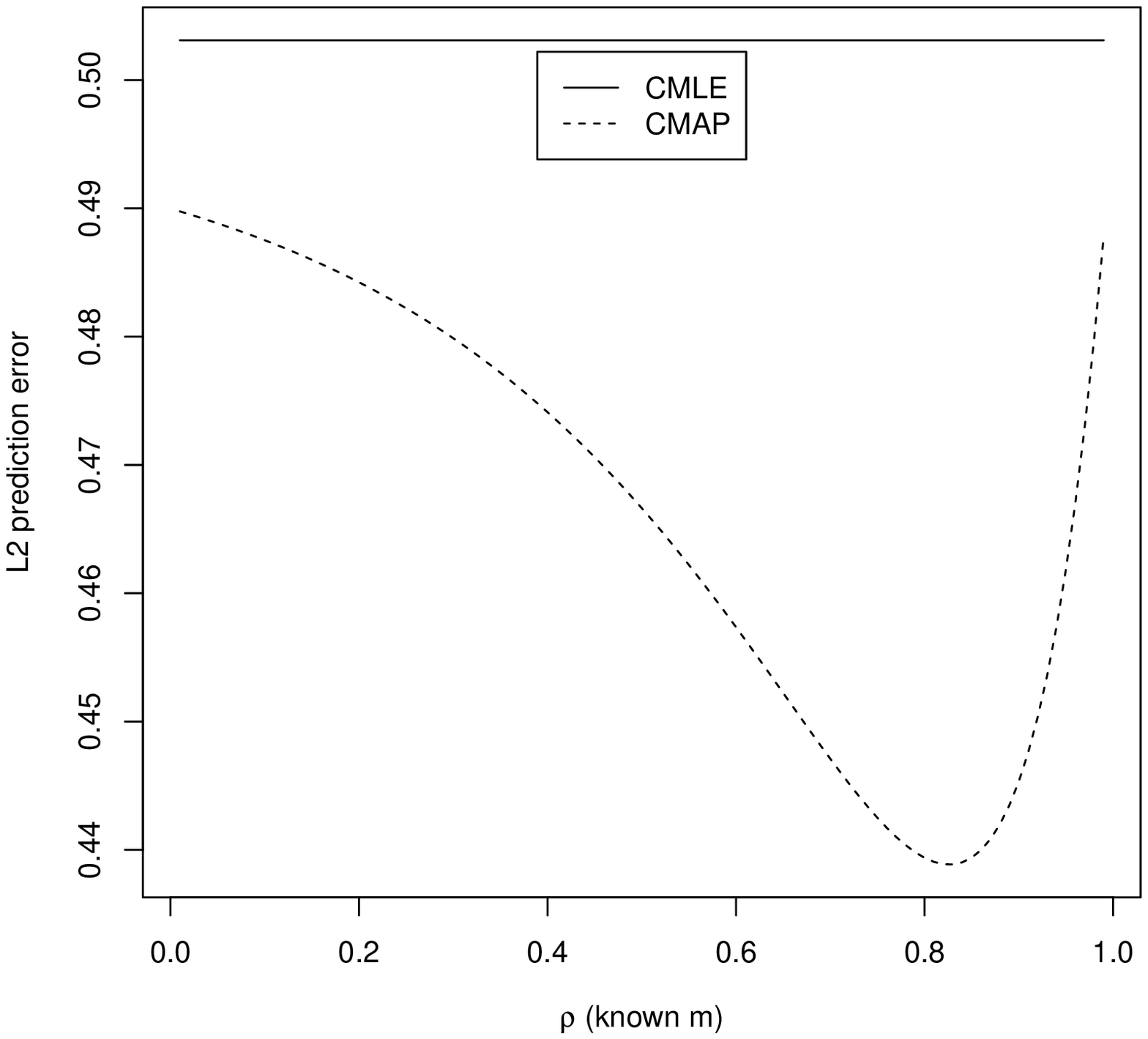}
\includegraphics[height=6.2cm]{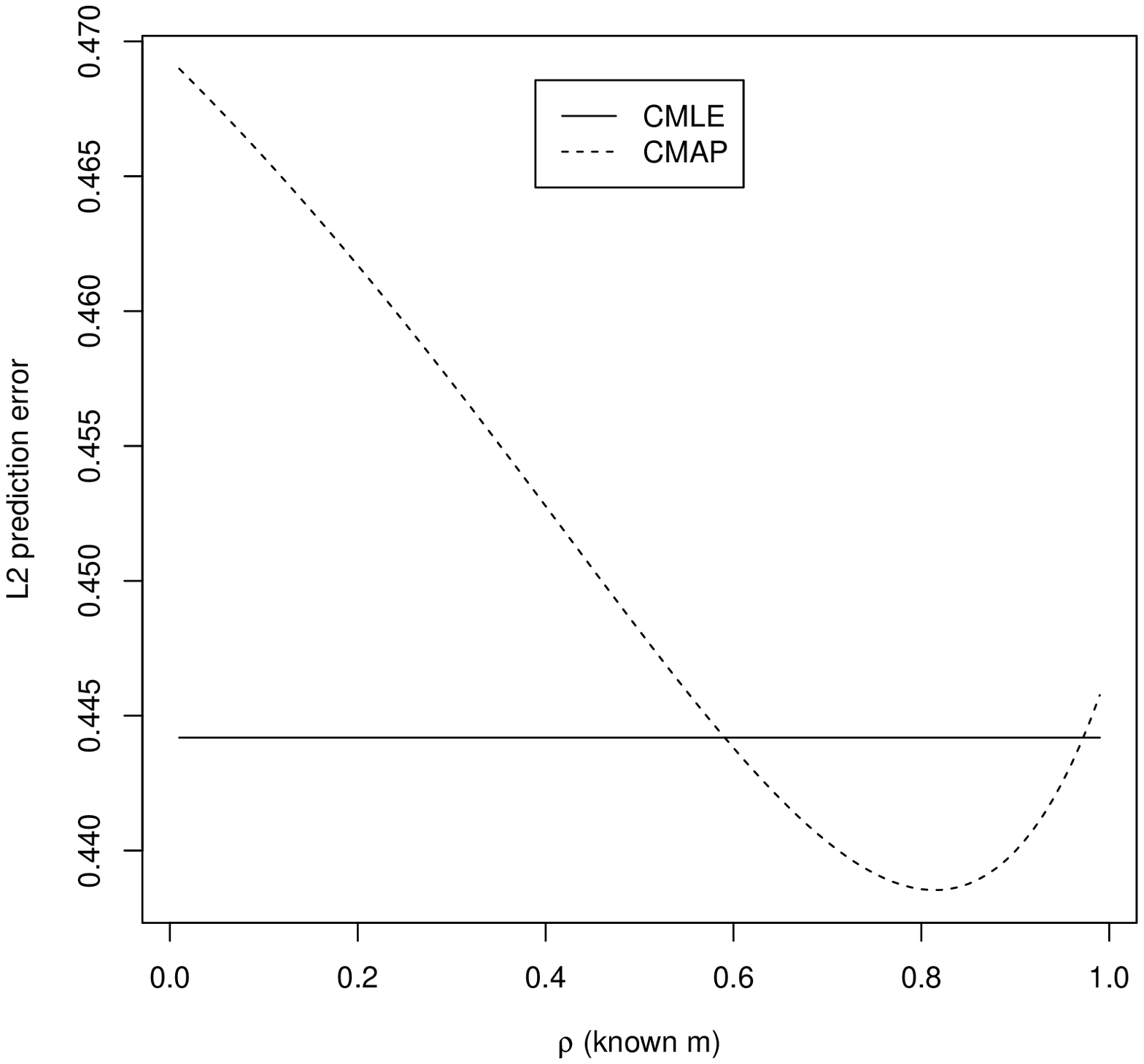}
\caption{{\small $L^2$ prediction error for $\rho=e^{-\tht\de}$ unknown ($\tht=1$, $\de=0.1$) and  $m=5$ (known) in terms of $\rho_0$ with ${\cal N}(\rho_0,10^{-2})$ prior: CMLE (plain horizontal), Bayes predictor (dashed). On the left : $n=30 \;(S=3,S+h*\de=4)$, on the right: $n=100 \;(S=10,S+h*\de=11)$.  }}
\end{figure}

\begin{table}[t]
\caption{$L^2$-prediction error ($\tht$ unknown) for MLE predictor and percentage variation of $L^2$-prediction error for others in the case where  $H=1$, $\de=0.1$, $v^2=0.01$, and $\rho_0\in\{0.5,0.75,0.85,0.9\}$.}\label{tab6}
\renewcommand{\arraystretch}{1.2}
\scalebox{0.68}{\begin{tabular}{|c||c|c|c|c||c|c|c|c||c|c|c|c|}
\hline
\backslashbox{$\mathbf{n=20}$}{$\mathbf \tht$} &\multicolumn{4}{|c||}{$\mathbf{0.5}$}&\multicolumn{4}{|c||}{$\mathbf{1}$}&\multicolumn{4}{|c|}{$\mathbf{2}$}\\
\hline
\bf CMLE&  \multicolumn{4}{|c||}{0.83}&\multicolumn{4}{|c||}{0.503}&\multicolumn{4}{|c|}{0.266}\\
\hline
$\mathbf \rho_0$ & \bf 0.5 & \bf 0.75 & \bf 0.85 & \bf 0.9  & \bf 0.5 & \bf 0.75 & \bf 0.85 & \bf 0.9  & \bf 0.5 & \bf 0.75 & \bf 0.85 & \bf 0.9 \\
\hline
\bf `Bayes' (\%) &  -8.41 &-14.67& -14.90& -13.55&   -7.25& -12.05 &-12.66& -11.48&  -5.82 &-6.56& -5.42& -3.21
\\ \hline\hline
\backslashbox{$\mathbf{n=100}$}{$\mathbf \tht$} &\multicolumn{4}{|c||}{$\mathbf{0.5}$}&\multicolumn{4}{|c||}{$\mathbf{1}$}&\multicolumn{4}{|c|}{$\mathbf{2}$}\\
\hline
\bf CMLE&  \multicolumn{4}{|c||}{0.679}&\multicolumn{4}{|c||}{0.444}&\multicolumn{4}{|c|}{0.242}\\
\hline
$\mathbf \rho_0$ & \bf 0.5 & \bf 0.75 & \bf 0.85 & \bf 0.9  & \bf 0.5 & \bf 0.75 & \bf 0.85 & \bf 0.9  & \bf 0.5 & \bf 0.75 & \bf 0.85 & \bf 0.9 \\
\hline
\bf `Bayes' (\%) &  3.76 &-0.25 &-1.04& -1.15& 0.89 &-1.13 &-1.22& -0.95& -0.44 &-0.77 &-0.46& -0.03
\\ \hline\hline
\end{tabular}}\end{table}

In Table~\ref{tab4}, we compare the obtained errors with varying values of $\de$, while in Table~\ref{tab5} the influence of $\tht$ is measured. First it appears that, obtained errors depend only on $S=n\de$, and not on the individual values of $n$ and $\de$ (see the bold type errors). It is not a surprise since examination of $L^2$-risks shows that leading terms are of order $n\de$ for each estimators. Moreover, the errors are much larger as $\de$ and|or $\tht$ are small. Again, it agrees with our theoretical framework since more $\de$ is small, more important is the correlation,  implying a degradation of the overall risk. Also, low values of $\tht$ corresponds to variables with high variance ($\var(X_1) = (2\tht)^{-1}$), and prediction is more difficult in this case. Finally, errors are represented in term of $n$ in Figure~7.2 (left): not surprisingly, errors decrease and estimators are asymptotically equivalent.

Concerning prediction when $\tht$ is unknown ($m$ known), we have computed the two predictors derived from the estimators given by \eqref{e77} (CMLE) and \eqref{e78} (`Bayes').  Figure~7.2 (right) that errors decrease with $n$ and Bayesian predictors are much better for small values of $n$. A noteworthy result is that  errors are significantly improved for any choice of prior, at least for $n$ small: see Table~\ref{tab6} for $n=20$ and Figure~7.3 (left) for $n=30$. This last conclusion may be tempered by the possibly bad behaviour of the CMLE in this framework. Finally for $n=100$, the Bayesian predictor is more sensitive to the prior (Figure~7.3, right).



\bibliographystyle{spbasic}
\bibliography{bb2013}
\end{document}